\documentclass[12pt]{amsart}
\usepackage{amscd,amsmath,bm,latexsym,amssymb,amscd,amsthm,amsfonts}
\usepackage[mathscr]{euscript} 
\usepackage[all]{xy}
\usepackage{tikz}
\usetikzlibrary{patterns,cd}
\usepackage{layout}
\usepackage{hyperref} 
\usepackage{enumerate}
\usepackage{dsfont}
\usepackage{caption} 
\newtheorem{thmintro}{Theorem}

\newtheorem{question}[thmintro]{Question}

\newtheorem{theorem}{Theorem}[section]
\newtheorem{corollary}[theorem]{Corollary}
\newtheorem{lemma}[theorem]{Lemma}
\newtheorem{proposition}[theorem]{Proposition}

\theoremstyle{definition}
\newtheorem{construction}[theorem]{Construction}
\newtheorem{definition}[theorem]{Definition}
\newtheorem{example}[theorem]{Example}
\theoremstyle{remark}
\newtheorem{remark}[theorem]{Remark}

 

\def\g{\mathfrak{g}}
\def\t{\mathfrak{t}}

\def\z{\mathfrak{z}}

\newcommand{\R}{\mathbb{R}}

\newcommand{\Z}{\mathbb{Z}}

\newcommand{\C}{\mathbb{C}}

\newcommand{\N}{\mathbb{N}}
\newcommand{\F}{\mathbb{F}}
\newcommand{\Q}{\mathbb{Q}}

\newcommand{\Hom}{\operatorname{Hom}}

\def\P{\mathcal{P}}
\def\g{\mathfrak{g}}
\def\z{\mathfrak{z}}
\def\u{\mathfrak{u}}
\newcommand{\uf}{\mathfrak{u}}

\def\t{\mathfrak{t}}
\newcommand{\BONE}{\mathds 1}

\newcommand{\co}{\colon\thinspace}

\begin{document}

\title[A stable splitting for spaces of commuting elements]
{A stable splitting for spaces of commuting elements in unitary groups}

\author[A.~Adem]{Alejandro Adem}
\address{Department of Mathematics\\
University of British Columbia, Vancouver BC V6T 1Z2, Canada}
\email{adem@math.ubc.ca}

\author[J.~M.~G\'omez]{Jos\'e Manuel G\'omez}
\address{ Departamento de Matem\'aticas\\
Universidad Nacional de Colombia sede Medell\'in, 
Medell\'in, Colombia }       
\email{jmgomez0@unal.edu.co}

\author[S.~Gritschacher]{Simon Gritschacher}
\address{Department of Mathematics\\
University of Munich, Munich, Germany }       
\email{simon.gritschacher@math.lmu.de}

\begin{abstract}
We prove an analogue of Miller's stable splitting of the unitary group $U(m)$ for spaces of commuting elements in $U(m)$. After inverting $m!$, the space $\Hom(\Z^n,U(m))$ splits stably as a wedge of Thom-like spaces of bundles of commuting varieties over certain partial flag manifolds. Using Steenrod operations we prove that our splitting does not hold integrally. Analogous decompositions for symplectic and orthogonal groups as well as homological results for the one-point compactification of the commuting variety in a Lie algebra are also provided.
\end{abstract}

\maketitle

\section{Introduction}\label{sec:introduction}

\subsection{Statement of results} In \cite{Miller} Miller showed that the unitary group $U(m)$ admits a stable splitting
\[
U(m)_+\simeq \bigvee_{0\leq k \leq m} \textnormal{Gr}_k(\C^m)^{\mathfrak{u}_k}
\]
where
\[
\textnormal{Gr}_k(\C^m)^{\mathfrak{u}_k}=(U(m)/U(m-k))_+\wedge_{U(k)} \mathfrak{u}_k^+
\]
is the Thom space of the adjoint bundle, i.e., of the vector bundle over the Grassmannian $\textnormal{Gr}_k(\C^m)$ associated with the adjoint representation of $U(k)$. In this paper we prove an analogous stable splitting for the space of commuting $n$-tuples in $U(m)$,
\[
\Hom(\Z^n,U(m))=\{(x_1,\dots,x_n)\in U(m)^n\mid x_i x_j=x_jx_i\textnormal{ for all }1\leq i,j\leq n\}\,,
\]
which holds after inverting $m!$. When $n\ge 2$ the spaces $\Hom(\Z^{n},U(m))$ may have torsion 
at primes dividing $m!$ (see \cite{KT} for the case of $SU(m)$). The presence of this torsion makes it necessary to adapt the proof of 
the stable splitting for $U(m)$ to obtain a splitting for $\Hom(\Z^{n},U(m))$ when $n\ge 2$.  

In our splitting the role of the adjoint bundle is played by a bundle of commuting varieties 
over a generalised Grassmannian. We denote by
\[
C_{n}(\uf_{k})=\{(X_{1},\dots,X_{n})\in \uf_{k}^{n} ~~|~~ [X_{i},X_{j}]=0 
\text{ for all }  1\le i,j\le n\}
\]
the space of commuting $n$-tuples in the Lie algebra $\uf_{k}$. When $k>1$ and $n>1$ the commuting 
variety $C_{n}(\uf_{k})$ is not a vector space as it is not closed under addition. The group 
$U(k)$ acts on $C_{n}(\uf_{k})$  diagonally by the adjoint representation, and this action extends to one on the one-point compactification $C_n(\mathfrak{u}_k)^+$. Our splitting is indexed by a certain poset of partitions of $m$ denoted $\mathcal{P}$. The elements $\lambda=(\lambda_a)_{a\in I}\in \mathcal{P}$ are partitions of $m$ into $2^n$ parts indexed by the set of binary sequences $I=\{0,1\}^n$. Given $a\in I$ we write $|a|=\sum_{i=1}^n a(i)$.

\begin{thmintro}\label{thm:main1}
After inverting $m!$ there is a stable splitting for every integer $n\ge 2$,
\[
\Hom(\Z^n,U(m))_{+}\simeq \bigvee_{(\lambda_a)_{a\in I}\in \mathcal{P}} 
\left(U(m)_+\wedge_{\prod_{a\in I} U(\lambda_a)} \bigwedge_{a\in I} C_{|a|}(\mathfrak{u}_{\lambda_a})^+ \right)\, .
\]
\end{thmintro}
The stable summands are Thom-like spaces for bundles of commuting varieties. These commuting varieties are irreducible, and in the non-abelian case the bundles that appear are not vector bundles, but rather bundles of infinite subspace arrangements.

Our methods apply to prove an analogous stable splitting for spaces of commuting elements in the compact symplectic group $Sp(m)$ (Theorem \ref{thm:splittingspm}), but not for $SU(m)$. This is because our proof uses a stable splitting of a maximal torus which is equivariant for the Weyl group action, and we do not know if such a splitting exists in the case of $SU(m)$. For orthogonal groups our methods apply; but because the orthogonal groups have abelian subgroups that are not contained in a maximal torus, we obtain a stable splitting only of the path-component of $\Hom(\Z^n,O(m))$ containing the trivial homomorphism (Theorem \ref{thm:splittingsom}).

To put our theorem in context we recall that the first author and F. Cohen \cite{AC} have previously proved a stable splitting of the space $\Hom(\Z^n,G)$ for any compact Lie group $G$. Let $S_n(G)\subseteq \Hom(\Z^n,G)$ be the subspace of those $n$-tuples of commuting elements of which at least one coordinate is the neutral element $1_G\in G$. Then there is a stable splitting

\begin{equation*} \label{eq:acsplitting}
\Hom(\Z^{n},G)_+\simeq  \bigvee_{0\le r\le
n}\left(\bigvee^{\binom{{n}}{{r}}} \Hom(\Z^{r},G)/S_{r}(G)
\right).
\end{equation*}
This splitting holds without inverting primes, but with the exception of the trivial cases and the cases $G=SU(2)$ and $G=SO(3)$, the stable summands have not been identified geometrically. For $G=U(m)$ and with $m!$ inverted, our decomposition is finer than the above, in that there is a stable equivalence
\[
\Hom(\Z^n,U(m))/S_n(U(m))\simeq \bigvee_{(\lambda_a)_{a\in I}\in \mathcal{S}} \left(U(m)_+\wedge_{\prod_{a\in I} U(\lambda_a)} \bigwedge_{a\in I} C_{|a|}(\mathfrak{u}_{\lambda_a})^+ \right)
\]
for a certain subset $\mathcal{S}\subseteq \mathcal{P}$ (Corollary \ref{cor:summandsinAC}).

Miller's stable splitting is modelled on a filtration of $U(m)$ in which the top filtration quotient is $\mathfrak{u}_m^+$. Similarly, our splitting is obtained from a filtration with top filtration quotient $C_n(\mathfrak{u}_m)^+$. The next theorem shows that when $n\geq 2$, certain primes \emph{must} be inverted for the top filtration quotient to split off stably. Using Steenrod operations we prove:

\begin{thmintro} \label{thm:main2}
Let $p$ be a prime and let $n\geq 2$ be an integer. The projection
\[
\Hom(\Z^n,U(p))_+\to C_n(\mathfrak{u}_{p})^+
\]
does not have a stable section up to homotopy at the prime $p$.
\end{thmintro}
In contrast, it should be noted that Crabb \cite{Crabb} and Baird, Jeffrey and Selick \cite{BJS} have obtained a stable decomposition of $\Hom(\Z^n,SU(2))$ in which the space $C_n(\mathfrak{su}_2)^+$ appears as a stable summand.\footnote{In Crabb's notation this is $C(T,\emptyset)$ with $|T|=n$.} In fact, this decomposition coincides with the aforementioned one of Adem and Cohen applied to $G=SU(2)$. Our splitting of $\Hom(\Z^n,Sp(m))$ applied to $Sp(1)=SU(2)$ reduces to theirs, but an ad-hoc argument is necessary to see that this splitting holds without inverting $2$.

A motivation for proving the stable splitting lies in the effort of computing the homology of $\Hom(\Z^n,U(m))$ which is still unknown. The homology of $C_n(\mathfrak{u}_m)^+$ appears much more tractable than that of $\Hom(\Z^n,U(m))$, although explicit computations are difficult even when $m$ is small. Unfortunately, inverting $m!$ has the effect of killing all torsion in $\Hom(\Z^n,U(m))$. But even if Theorem \ref{thm:main1} fails to hold without inverting primes, it remains to be determined if it holds at the level of integral homology. Preliminary calculations indicate that this is the case when $n=2$ and $m=2,3$.
\begin{question}
Is there an isomorphism
\[
\tilde{H}_{*}(\Hom(\Z^n,U(m))_{+};\Z)\cong \bigoplus_{(\lambda_a)_{a\in I}\in \mathcal{P}} 
\tilde{H}_{*}\left(U(m)_+\wedge_{\prod_{a\in I}U(\lambda_a)} \bigwedge_{a\in I} C_{|a|}(\mathfrak{u}_{\lambda_a})^+;\Z\right)? 
\]
\end{question}
This question may be related to the Steenrod splitting of the homology of symmetric products \cite[Section 22]{Steenrod}; indeed, the quotient of $C_n(\mathfrak{u}_m)^+$ by the adjoint action of $U(m)$ is a symmetric smash product of spheres, and a key step in the proof of Theorem \ref{thm:main2} is the comparison of $\Hom(\Z^n,U(p))$ with a symmetric product of spheres. In fact, the moduli space of representations, i.e., the quotient of $\Hom(\Z^n,U(m))$ by the conjugation action of $U(m)$ is the $m$-fold symmetric product of the $n$-torus, $\textnormal{SP}^m((S^1)^n)$. Rationally then, the stable factors in our splitting correspond precisely to those in a rational stable splitting of the symmetric product.

This article has an appendix in which we record basic results about the compactified commuting variety $C_n(\mathfrak{g})^+$ where $\mathfrak{g}$ is the Lie algebra of a compact connected Lie group $G$. These results are used frequently throughout the earlier sections, but we have decided to concentrate them in the appendix which is self-contained and may be of independent interest. Amongst other results we will prove the following:

\begin{thmintro} \label{thm:main3}
Let $G$ be a compact connected Lie group of rank $r\geq 1$ and dimension $d$, and let $W$ denote the Weyl group of $G$. For every $n\geq 1$, $C_{n}(\mathfrak{g})^+$ is a $\Z[1/|W|]$-homology sphere of dimension $nr$ if $n$ is even, and of dimension $d+(n-1)r$ if $n$ is odd.
\end{thmintro}
A computation of the mod-$p$ cohomology of $C_n(\mathfrak{u}_p)^+$ when $n>2$ is even is also included.

\subsection{Organisation of the paper}

In Section \ref{sec:Miller} we review Miller's 
stable splitting for the unitary groups $U(m)$. In Section \ref{sec:filtration} we 
introduce the poset $\mathcal{P}$ and construct the filtration of $\Hom(\Z^n,U(m))$ which will lead to the stable splitting of Theorem \ref{thm:main1}. The theorem is proved in Section \ref{sec:splitting}. The non-splitting result, Theorem \ref{thm:main2}, is proved in Section \ref{sec:nonsplitting}. The article concludes with Appendix \ref{sec:commutingvariety} which contains in particular the proof of Theorem \ref{thm:main3}.

\section{The stable splitting for the unitary group}\label{sec:Miller}

In this section we briefly review Miller's stable splitting for the group $U(m)$. We follow the 
presentation of Crabb \cite{Crabb2}.

We start by constructing a suitable filtration of $U(m)$. For every $0\le k\le m$, let
\[
F^{k}(U(m))=\{A\in U(m)\mid \dim(\ker(A-\mathds{1}))\geq m-k\}\, .
\]
In other words, the elements in $F^{k}(U(m))$ 
are the matrices that have $1$ as an eigenvalue with multiplicity at least $m-k$. This way we obtain 
an increasing filtration of $U(m)$
\begin{equation}\label{filtration U(m)}
\{\mathds{1}\}=F^{0}(U(m))\subseteq F^{1}(U(m)) \subseteq\cdots \subseteq F^{m}(U(m))=U(m).
\end{equation}
Alternatively, this filtration is obtained in the following way. Let $T\subseteq U(m)$ be the maximal torus consisting of diagonal matrices and let $F^{k}(T)$ be the subspace of $T\cong (S^1)^m$ with at least $m-k$ coordinates equal to $1$. Consider the map
\begin{alignat*}{1}
\tilde{\varphi}\co U(m)\times T & \to U(m) \\
(g,t)&\mapsto gtg^{-1}\, . 
\end{alignat*}
Then $F^k(U(m))=\tilde{\varphi}(U(m)\times F^k(T))$. We will take this point of view in the next section to construct filtrations of $\Hom(\Z^n,U(m))$ from filtrations of $T^n$.

Let $V_{k}(\C^m)=U(m)/U(m-k)$ denote the Stiefel 
manifold of orthonormal $k$-frames in $\C^{m}$ and let $\textnormal{Gr}_{k}(\C^{m})=V_k(\C^m)/U(k)$ be the Grassmannian. Let $\u_{k}$ denote the Lie algebra of $U(k)$. It is the space of $k\times k$ skew-Hermitian matrices with the commutator bracket. We next explain how the stratum
\[
F_{k}(U(m)):=F^{k}(U(m))\backslash F^{k-1}(U(m))
\]
can be 
identified with the total space of the vector bundle 
\begin{equation} \label{eq:adjointbundle}
\u_{k}\to V_{k}(\C^m)\times_{U(k)} \u_{k}\to \textnormal{Gr}_{k}(\C^{m})
\end{equation}
associated via the adjoint representation.

\begin{definition} \label{def:cayley}
The \emph{Cayley transform} is the map
\begin{alignat*}{2} 
\psi_m\co & \uf_m &&\to  F_m(U(m))\\
& X && \mapsto  (X-\mathds{1})(X+\mathds{1})^{-1}\, .
\end{alignat*}
\end{definition}

Observe that the space $F_{m}(U(m))$ is an open dense subspace of $U(m)$ on which $U(m)$ acts by conjugation. The following lemma is basic:
\begin{lemma} \label{lem:cayleydiffeo}
The map $\psi_m\co \mathfrak{u}_m\to F_m(U(m))$ is a $U(m)$--equivariant diffeomorphism.
\end{lemma}
\begin{proof}
The inverse of $\psi_m$ is given by the formula $\psi_m^{-1}(A)=(\mathds{1}-A)^{-1}(\mathds{1}+A)$.
\end{proof}

Thus, when $k=m$ the Cayley transform identifies $F_{k}(U(m))$ with the total space of the adjoint bundle (\ref{eq:adjointbundle}). More generally, for $0\le k\le m$ the stratum $F_{k}(U(m))$ fibres over $\textnormal{Gr}_{k}(\C^{m})$ by considering
the map 
\begin{alignat*}{1} 
F_{k}(U(m)) &\to  \textnormal{Gr}_{k}(\C^{m}) \\
A & \mapsto  \ker(A-\mathds{1})^{\perp}\, .
\end{alignat*}
The point is that a transformation $A\in F_{k}(U(m))$ can be written in the form 
\[
\mathds{1}\oplus A|_{V_1}:V_{0}\oplus V_{1}\to V_{0}\oplus V_{1}
\]
where  $V_{0}=\ker(A-\mathds{1})$ and $V_{1}=V_0^{\perp}$. So $A$ is uniquely determined by $V_1\in \textnormal{Gr}_{k}(\C^m)$ and the restriction $A|_{V_1}$. Note that $A|_{V_1}$ is an element of $F_k(U(V_1))\cong \mathfrak{u}_k$. More precisely, the map 

\begin{alignat*}{1} 
V_k(\mathbb{C}^m)\times_{U(k)} \u_{k} & \to F_{k}(U(m)) \\
[(\bar{g},X)] & \mapsto g(\mathds{1}\oplus \psi_k(X))g^{-1}
\end{alignat*}
is a diffeomorphism over $\textnormal{Gr}_{k}(\C^{m})$ (here $g\in U(m)$ is a representative of an element $\bar{g}\in U(m)/U(m-k)= V_k(\C^m)$). It follows that the 
subquotient $F^{k}(U(m))/F^{k-1}(U(m))$ can be identified with the Thom space of the adjoint bundle:
\[
\textnormal{Gr}_{k}(\C^{m})^{\uf_{k}}=(U(m)/U(m-k))_+\wedge_{U(k)} \mathfrak{u}_k^+\,.
\] 
In particular, the top filtration quotient $U(m)/F^{m-1}(U(m))$ is $\mathfrak{u}_m^+$.

Finally, one shows that the filtration (\ref{filtration U(m)}) 
splits stably. The top quotient is split off using a Pontryagin-Thom construction: Let $\mathcal{H}(m)$ be the vector space of hermitian matrices of size $m\times m$. The embedding
\begin{align*}
U(m)\times \mathcal{H}(m)&\to \textnormal{Mat}_m(\C)\\
(A,Z)&\mapsto A\exp(-Z)
\end{align*}
is onto $GL_m(\C)$, and so $GL_m(\C)$ is a tubular neighbourhood of $U(m)$ inside $\textnormal{Mat}_m(\C) \cong \u_{m}\times \mathcal{H}(m)$. Collapsing the complement to a point we obtain a map 
\[
\u_{m}^{+}\wedge \mathcal{H}(m)^{+}\to U(m)_{+}\wedge \mathcal{H}(m)^{+}. 
\]
Note that $\mathcal{H}(m)^+$ is a sphere, and so this is a stable map $\u_{m}^+\to U(m)_+$. One checks that this is a stable section up to homotopy of the projection
\begin{equation} \label{eq:millertopstratum}
U(m)_{+}\to U(m)/F^{m-1}(U(m))\cong \u_{m}^{+}\, .
\end{equation}

The general case, which shows that $\textnormal{Gr}_k(\C^m)^{\mathfrak{u}_k}$ splits off stably from $F^k(U(m))_+$, is handled similarly using the Pontryagin-Thom 
construction fibrewise. For this one notes that the stable section constructed above is $U(m)$-equivariant. Taking the wedge over the various sections $\textnormal{Gr}_k(\C^m)^{\mathfrak{u}_k}\to F^k(U(m))_+\hookrightarrow U(m)_+$ leads to a stable map
\[
\bigvee_{0\le k\le m}\textnormal{Gr}_{k}(\C^{m})^{\uf_{k}}\to U(m)_+
\]
which is a homotopy equivalence. The details are in \cite{Crabb2}.

One may attempt the same approach to construct a stable splitting of $\Hom(\Z^n,U(m))$. However, the Pontryagin-Thom collapse does not apply naively in that situation, because the open embedding $U(m)\times \mathcal{H}(m)\hookrightarrow \textnormal{Mat}_m(\C)$ considered above does not produce a tubular neighbourhood of any sort for $\Hom(\Z^n,U(m))$.

\section{The filtration of $\Hom(\Z^n,U(m))$}\label{sec:filtration}

The aim of this section is to describe the filtration of $\Hom(\Z^n,U(m))$ on which our stable 
splitting is based. Throughout this section let $n\geq 1$ and $m\geq 0$ be fixed.

\subsection{Commuting elements}

The space $\Hom(\Z^n,U(m))$ of $n$-tuples of pairwise commuting elements in $U(m)$ (defined in Section \ref{sec:introduction}) is topologized as a subspace of $U(m)^n$. Let $T\subseteq U(m)$ 
be the maximal torus consisting of all diagonal matrices. Consider the map
\begin{alignat*}{1}
\tilde{\varphi}\co U(m)\times T^n& \to \Hom(\Z^n,U(m)) \\
(g,t_1,\dots,t_n) & \mapsto (gt_1 g^{-1},\dots,gt_ng^{-1})\, .
\end{alignat*}
It is easily seen to be surjective, as any abelian subgroup of $U(m)$ is contained in a maximal torus. 
There is an action of the normaliser $N_{U(m)}(T)$ on $U(m)\times T^n$ defined by
\[
h \cdot (g,t_1,\dots,t_n)=(gh^{-1},ht_1h^{-1},\dots,h t_n h^{-1})\, ,\quad h\in N_{U(m)}(T)\, .
\]
The map $\tilde{\varphi}$ is invariant under this action and so it descends to a map
\begin{equation} \label{eq:actionmapliegroup}
\varphi\co U(m)\times_{N_{U(m)}(T)} T^n   \to \Hom(\Z^n,U(m))\, .
\end{equation}
Notice that
\[
U(m)\times_{N_{U(m)}(T)} T^n=U(m)/T\times_{\Sigma_m} T^n\,
\]
where the  symmetric group $\Sigma_m$ appears as the Weyl group $N_{U(m)}(T)/T$. It acts on $T^n$ by 
permuting the $m$ factors of $T\cong (S^1)^m$.

\subsection{Indexing poset} \label{sec:poset}
A filtration of $\Hom(\Z^n,U(m))$ can be constructed by starting with a $\Sigma_m$-invariant 
filtration of $T^n$ and pushing it forward via $\varphi$. The filtration of $T^n$ that we shall 
consider is indexed over a poset $\mathcal{P}$ of partitions of $m$ which we will now describe. 
Let $I=\{0,1\}^n$ be the set of binary sequences of length $n$, and let $\mathcal{P}$ be the set 
of $I$-indexed ordered integer partitions of $m$, i.e.,
\[
\mathcal{P}= \left\{(\lambda_a)_{a \in I} \in \Z^I ~\left|~ \lambda_a\geq 0 
\textnormal{ for all }a\in I \textnormal{ and } \sum_{a\in I} \lambda_a=m \right\}\right.\,.
\]
For example, when $n=2$, elements in $\P$ are of the form
$\lambda=(\lambda_{(0,0)}, \lambda_{(0,1)}, \lambda_{(1,0)}, \lambda_{(1,1)})$, where 
each $\lambda_{a}$ is a non-negative integer and  
\[
\lambda_{(0,0)}+\lambda_{(0,1)}+\lambda_{(1,0)}+\lambda_{(1,1)}=m.
\] 

The set $I=\{0,1\}^n$ is totally ordered by the lexicographic order in which 
$(0,\dots,0)$ is the smallest element and $(1,\dots,1)$ is the largest element. With this in mind, 
to each partition $\lambda= (\lambda_a)_{a \in I}$ we can associate an $m\times n$  
matrix $M(\lambda)$ with entries in $\{0,1\}$ defined by blocks in the following way. 
In the first $\lambda_{(0,\dots,0)}$ rows of $M(\lambda)$ we put the row $(0,\dots,0)$. 
Subsequently, we define the next $\lambda_{(0,\dots,0,1)}$ rows in $M(\lambda)$ to be 
equal to $(0,\dots,0,1)$. We continue this way, using the order in $I$, until we reach 
the last $\lambda_{(1,\dots,1)}$ rows in $M(\lambda)$ and these rows are defined 
to be equal to $(1,\dots,1)$.  Note that the assignment $\lambda\mapsto M(\lambda)$ is one-to-one.

\begin{example} 
Consider the particular case 
$n=2$ and $m=6$, and let $\lambda$ be the $I$-indexed partition of $m=6$ given by 
\[
\lambda_{(0,0)}=1,\ \lambda_{(0,1)}=1,\ \lambda_{(1,0)}=2,\ \lambda_{(1,1)}=2.
\]
In this case, the matrix $M(\lambda)$ associated to $\lambda$ is the $6\times 2$ matrix 
given by
\[
    M(\lambda)= 
    \begin{pmatrix}
              0 & 0 \\ \hline 
				0 & 1 \\ \hline 
				1 & 0 \\ 
				1 & 0 \\ \hline 
				1 & 1\\
				1 & 1
        \end{pmatrix}.
\]
\end{example}

Using these matrices $\{M(\lambda)~|~ \lambda\in \mathcal{P}\}$ we can define a partial order on 
$\mathcal{P}$ in the following way: Given $\lambda,\mu\in \mathcal{P}$ we declare $\mu\leq \lambda$ 
if, after suitable permutations of rows, the matrix $M(\mu)$ is obtained from $M(\lambda)$ by replacing 
some entries that are equal to $1$ with a $0$. 

\begin{example} 
Consider the case $n=2$ and $m=6$, and let 
$\mu, \lambda$ be the $I$-indexed partitions of $m=6$ defined by 
\begin{align*}
\mu_{(0,0)}&=1,\ \mu_{(0,1)}=2,\ \mu_{(1,0)}=2,\ \mu_{(1,1)}=1,\\
\lambda_{(0,0)}&=1,\ \lambda_{(0,1)}=1,\ \lambda_{(1,0)}=2,\ \lambda_{(1,1)}=2.
\end{align*}
In this example the matrices $M(\mu)$ and $M(\lambda)$ are as follows
\[
 M(\mu)= 
    \begin{pmatrix}
        0 & 0 \\ \hline
				0 & 1 \\  
				0 & 1 \\ \hline 
				1 & 0 \\ 
				1 & 0\\ \hline
				1 & 1
      \end{pmatrix},
		\ \ \ \ \
		 M(\lambda)= 
    \begin{pmatrix}
        0 & 0 \\ \hline 
				0 & 1 \\ \hline 
				1 & 0 \\ 
				1 & 0 \\ \hline 
				1 & 1\\
				1 & 1
             \end{pmatrix}.
\] 
In this case, $M(\mu)$ is obtained from $M(\lambda)$ by replacing the entry 
$(5,1)$ with a $0$ and then permuting rows $3$ and $5$. This shows that 
$\mu\leq \lambda$.
\end{example}

We endow the poset $\mathcal{P}$ with the upper topology in which the closed 
sets are generated by the subbasis $\{\mathcal{P}_{\leq \lambda} \mid  \lambda \in \mathcal{P} \}$ 
where $\mathcal{P}_{\leq \lambda}=\{\mu\in \mathcal{P} \mid \mu\leq \lambda\}$. 
If $X$ is a space and $f\co X\to \mathcal{P}$ is a continuous map, then for all $\lambda\in \mathcal{P}$ the subspace
\[
F^\lambda (X):=f^{-1}(\mathcal{P}_{\leq \lambda}) \subseteq X
\]
is closed, and there is an inclusion $F^\mu(X)\subseteq F^\lambda(X)$ whenever $\mu\leq \lambda$. Thus, $f$ induces a filtration of $X$ by closed subspaces indexed by $\mathcal{P}$.

 \subsection{Construction of the filtration} \label{sec:construction}
To filter $T^n$ we define a map $\Psi_T \co T^n\to \mathcal{P}$ in the following way. 
An element $x\in T^n$ can be evidently viewed as a matrix of size $m\times n$ with entries in $S^1$. 
In this picture the Weyl group $\Sigma_m$ permutes the rows of $x$. To each row $x_i=(x_{i1},\dots,x_{in})\in (S^1)^n$, $1\leq i \leq m$ of $x$ we can assign a binary sequence $\textnormal{seq}(x_i)\in I$ by setting
\[
\textnormal{seq}(x_i)(j) =\begin{cases} 0 & \textnormal{if } x_{ij}=1, \\ 1 & \textnormal{if }x_{ij}\neq 1 \end{cases}
\]
for all $1\leq j \leq n$. For $a\in I$ we let $\lambda_a$ be the number of $1\leq i\leq m$ such that $\textnormal{seq}(x_i)=a$ and set $\Psi_T(x):=(\lambda_a)_{a\in I}$.

\begin{example}
Suppose that $n=2$ and $m=6$, and let $x\in T^{2}$ be the element corresponding to the matrix 
\[
x=
\begin{pmatrix}
* & * \\ 
1 & 1 \\  
* & 1 \\ 
* & 1 \\ 
* & *\\
1 & *
\end{pmatrix},
\] 
where the $*$`s can take any value in $S^{1}$ different from $1$. Then $\lambda=\Psi_T(x)\vdash 6$ is given by
\[
\lambda_{(0,0)}=1,\, \lambda_{(0,1)}=1,\, \lambda_{(1,0)}=2,\, \lambda_{(1,1)}=2\, .
\]
\end{example}

It is easily checked that for each $\lambda\in \mathcal{P}$ the subspace
\begin{equation*} \label{eq:filtrationoftorus}
F^\lambda(T^n):=\Psi_T^{-1}(\mathcal{P}_{\leq \lambda})\subseteq T^n
\end{equation*}
is closed, and thus $\Psi_T\co T^n\to \mathcal{P}$ is continuous. The filtration defined by $\Psi_T$ refines the fat wedge filtration of $T^n$.

\begin{example} \label{ex:fatwedget}
Suppose that $n=1$. Then $I=\{0,1\}$ and thus all partitions $\lambda\in \mathcal{P}$ are of the 
form $\lambda=(m-k,k)$ for some $0\leq k \leq m$. Projection onto the second component defines a 
homeomorphism of posets $\mathcal{P}\cong \N_{\leq m}$. Then $\Psi_T^{-1}(\N_{\leq k})$ is precisely the subspace of 
$T\cong (S^1)^m$ consisting of those $m$-tuples that have at least $m-k$ entries equal to $1$. Therefore, the filtration associated with $\Psi_T\co T\to \N_{\leq m}$ is the fat wedge filtration of $T$.
\end{example}

\begin{example} \label{ex:fatwedgetn}
More generally, for any $n\geq 1$ we can define a map of posets $f \co \mathcal{P}\to \N$ by 
sending $(\lambda_{a})_{a\in I}$ to $\sum_{a\in I} \lambda_a |a|$, where 
$|a|:=\sum_{i=1}^n a(i)$. For $x\in T^n$, $f(\Psi_T(x))$ is precisely the number of 
entries of $x$ that are different from $1$. Therefore, the filtration defined by $f \Psi_T$ is the fat wedge filtration of 
$T^n\cong (S^1)^{nm}$.
\end{example}

Since every commuting $n$-tuple in $U(m)$ is conjugate to an $n$-tuple in $T$, the quotient of $\Hom(\Z^n,U(m))$ by the conjugation action of $U(m)$ is homeomorphic to $T^n/\Sigma_m$. Clearly, $\Psi_T$ factors through a continuous map $\overline{\Psi}_T\co T^n/\Sigma_m \to \mathcal{P}$.

\begin{construction} \label{cons:construction}
We let
\[
\Psi\co \Hom(\Z^n,U(m))\to \mathcal{P}
\]
be the composition of the quotient map $\Hom(\Z^n,U(m))\to T^n/\Sigma_m$ with $\overline{\Psi}_T$, and let 
\[
F^\lambda (\Hom(\Z^n,U(m))) := \Psi^{-1}(\mathcal{P}_{\leq \lambda})
\]
be the filtration of $\Hom(\Z^n,U(m))$ induced by $\Psi$.
\end{construction}

Of course,
\[
F^\lambda (\Hom(\Z^n,U(m)))  = \varphi((U(m)/T\times_{\Sigma_m} F^\lambda(T^n)))\,,
\]
and so the filtration of $\Hom(\Z^n,U(m))$ is the filtration of $T^n$ pushed forward by $\varphi$.

We will frequently adjoin a basepoint to $\Hom(\Z^n,U(m))$ in which case we extend the filtration to include the basepoint at every filtration stage.

The filtration of $\Hom(\Z^n,U(m))$ generalises the one used in Miller's 
splitting and it refines the filtration used in the splitting of Adem and Cohen.

\begin{example} \label{ex:filtrationmiller}
Suppose that $n=1$. The filtration of 
$\Hom(\Z,U(m))\cong U(m)$ induced by $\Psi \co U(m)\to \mathcal{P}\cong \N_{\leq m}$ is given by
\[
\Psi^{-1}(\N_{\leq k})=\left\{g\in U(m) \mid \dim \ker(g-\mathds{1})\geq m-k\right\}\, .
\]
This is the filtration on which the stable splitting of $U(m)$ is based 
(see Section \ref{sec:Miller}).
\end{example}

\begin{example} \label{ex:filtrationademcohen}
Let $f\co \mathcal{P}\to \N$ be the map of posets defined by
\[
f(\lambda)=n-|\{i\in \{1,\dots,n\}\mid  \sum_{a\in I}\lambda_a a(i)=0\}|\, .
\]
Then the filtration of $\Hom(\Z^n,U(m))$ induced by $f \Psi \co \Hom(\Z^n,U(m))\to \N$ 
coincides with the one induced by the fat wedge filtration of $U(m)^n$,
\[
(f \Psi)^{-1}(\N_{\leq k}) = \{(g_1,\dots,g_n)\in \Hom(\Z^n,U(m)) \mid 
\textnormal{at least }n-k\textnormal{ of the }g_i\textnormal{ are equal to }\BONE\}\, .
\]
This is the filtration that is used in the stable splitting of Adem and Cohen.
\end{example}

\subsection{Subquotients of the filtration}
The map $\Psi \co \Hom(\Z^n,U(m))\to \mathcal{P}$ gives a partition of $\Hom(\Z^n,U(m))$ into ``strata" (these are not manifolds in general)
\[
\Hom(\Z^n,U(m))_\lambda := \Psi^{-1}(\lambda), \; \lambda \in \mathcal{P}
\]
that we will now describe.

From the introduction, recall the definition of the variety $C_n(\mathfrak{u}_m)$ of commuting elements in the Lie algebra $\mathfrak{u}_m$. We topologize $C_n(\mathfrak{u}_m)$ as a subspace of the Euclidean space $\mathfrak{u}_m^n$. Also recall that $I=\{0,1\}^n$ is totally ordered via the 
lexicographic order. By fixing this total order on $I$ we can view, for every 
$\lambda\in \mathcal{P}$, the direct product
\[
U(\lambda):=\prod_{a\in I} U(\lambda_a)
\]
as a subgroup of $U(m)$ in the obvious way. For $a\in I$ we write $|a|:=\sum_{i=1}^n a(i)$.

\begin{proposition} \label{prop:strata}
For every $\lambda\in \mathcal{P}$ there is a $U(m)$-equivariant homeomorphism
\[
\Hom(\Z^n,U(m))_\lambda \cong U(m) \times_{U(\lambda)}  \prod_{a\in I} C_{|a|}(\mathfrak{u}_{\lambda_a}) \, .
\]
\end{proposition}
\begin{proof}
According to \cite[Proposition II.3.2]{Bredon} it is enough to construct a $U(m)$-equivariant map
\[
p_\lambda\co \Hom(\Z^n,U(m))_\lambda\to U(m)/U(\lambda)
\]
and to identify the fibre over the neutral coset, as a $U(\lambda)$-space, with 
$\prod_{a\in I} C_{|a|}(\mathfrak{u}_{\lambda_a})$.

Recall that elements of $T^n$ can be viewed as $m\times n$ matrices with entries in $S^1$. 
Given $x\in T^n$, we can assign to each row 
$x_i=(x_{i1},\dots,x_{i n})\in (S^1)^n$, $1\leq i \leq m$ a binary sequence $\textnormal{seq}(x_i)\in I$ as in Section \ref{sec:construction}. Let $R_\lambda$ be the subspace of $\Psi^{-1}_T(\lambda)\subseteq T^n$ 
defined by
\[
R_\lambda=\{x\in \Psi^{-1}_T(\lambda) \mid \textnormal{seq}(x_1)\leq \cdots \leq \textnormal{seq}(x_m) \}\, .
\]
As the $\Sigma_m$-orbit of any $x\in \Psi_T^{-1}(\lambda)$ intersects $R_\lambda$ non-trivially, 
the restriction of the action map $\tilde{\varphi}\co U(m)\times T^n\to \Hom(\Z^n,U(m))$ to $U(m)\times R_\lambda$,
\[
\tilde{\varphi}|\co U(m)\times R_\lambda \to \Hom(\Z^n,U(m))_\lambda\,,
\]
is surjective. Thus, for every $(g_1,\dots,g_n)\in \Hom(\Z^n,U(m))_\lambda$ there is a 
$h=h(g_1,\dots,g_n)\in U(m)$ such that $(h^{-1}g_1h,\dots,h^{-1}g_n h)\in R_\lambda$. 
It is easy to see that $h$ is uniquely determined modulo $U(\lambda)$ by $(g_1,\dots,g_n)$, 
and thus the assignment
\begin{alignat*}{1}
p_\lambda\co \Hom(\Z^n,U(m))_\lambda & \to U(m)/U(\lambda) \\
(g_1,\dots,g_n) & \mapsto h(g_1,\dots,g_n)U(\lambda)
\end{alignat*}
is well-defined. As $\tilde{\varphi}|$ is a closed surjection, hence a quotient map, and 
$p_\lambda \tilde{\varphi}|\co U(m)\times R_\lambda\to U(m)/U(\lambda)$ is continuous 
being the obvious projection, $p_\lambda$ is continuous as well.
 
It remains to identify the fibre over $\BONE U(\lambda)$. Using $\oplus$ to denote the direct 
sum of matrices or representations, the elements of $p_\lambda^{-1}(\BONE U(\lambda))$ are 
precisely those $n$-tuples of the form
\[
(g_1,\dots,g_n)=\bigoplus_{a\in I} (g_1^{(a)},\dots,g_n^{(a)})
\]
where $(g_1^{(a)},\dots,g_n^{(a)}) \in \Hom(\Z^n,U(\lambda_a))$ satisfies for all $i\in \{1,\dots,n\}$
\[
g_i^{(a)}-\BONE=\begin{cases} 0 & \textnormal{if } a(i)=0 \\ \textnormal{non-singular} 
& \textnormal{if } a(i)=1\, . \end{cases}
\]
The Cayley transform
\begin{alignat*}{1}
\psi_k \co \mathfrak{u}_k & \to  \{ g\in U(k) \mid g-\BONE\textnormal{ is non-singular}\}
\end{alignat*}
from Definition \ref{def:cayley} respects commutativity, i.e., for all 
$X,Y\in \mathfrak{u}_k$ we have $[X,Y]=0$ if and only if $[\psi_k(X),\psi_k(Y)]=\BONE$. 
Therefore, we can define for every $a\in I$ a $U(\lambda_a)$-equivariant map
\[
\psi'_{\lambda_a}\co C_{|a|}(\mathfrak{u}_{\lambda_a}) \to \Hom(\Z^n,U(\lambda_a))
\]
in the following fashion. Let $\iota_a\co \{1,\dots ,|a|\}\to \{1,\dots,n\}$ be the unique 
monotone injection with $\textnormal{im}(\iota_a)=\{i\in \{1,\dots,n\}\mid a(i)=1\}$. Let
\begin{equation} \label{eq:shufflemap}
\textnormal{sh}_a\co \Hom(\Z^{|a|},U(\lambda_a))\to \Hom(\Z^n,U(\lambda_a))
\end{equation}
be the map defined by $\textnormal{sh}_a(g_1,\dots,g_{|a|})=(g_1',\dots,g_n')$ with
\[
g_i'=\begin{cases} \BONE & \textnormal{if } \iota_a^{-1}(i)=\emptyset \\ g_j & 
\textnormal{if } \iota_a(j)=i\, . \end{cases}
\]
Then $\psi'_{\lambda_a}$ is defined by
\[
\psi'_{\lambda_a}(X_1,\dots,X_{|a|})
=\textnormal{sh}_a(\psi_{\lambda_a}(X_1),\dots,\psi_{\lambda_a}(X_{|a|}))\, .
\]
These maps assemble into a $U(\lambda)$-equivariant map
\[
\prod_{a\in I} C_{|a|}(\mathfrak{u}_{\lambda_a}) 
\xrightarrow{\prod_{a\in I} \psi'_{\lambda_a}} 
\prod_{a\in I} \Hom(\Z^{n},U(\lambda_a)) \xrightarrow{\oplus} \Hom(\Z^n,U(m))
\]
which is easily seen to be a homeomorphism onto 
$p_\lambda^{-1}(\BONE U(\lambda))\subseteq \Hom(\Z^n,U(m))$.
\end{proof}

Observe that the homogeneous space $U(m)/U(\lambda)$ is the flag manifold consisting of all the orthogonal decompositions of $\C^{m}$
into vector subspaces of dimensions $\lambda_a$, $a\in I$. By Proposition \ref{prop:strata}, the space $\Hom(\Z^n,U(m))_\lambda$ is the total space of a fibre bundle
\[
\prod_{a\in I} C_{|a|}(\mathfrak{u}_{\lambda_a}) \to \Hom(\Z^n,U(m))_\lambda\to U(m)/U(\lambda)
\]
whose fibres are products of commuting varieties. This generalises naturally the case 
$n=1$ where, for $\lambda=(m-k,k)$, the space $U(m)_\lambda$ is the total space of the adjoint bundle
\[
V_k(\C^m)\times_{U(k)} \mathfrak{u}_k \to \textnormal{Gr}_{k}(\mathbb{C}^m)
\]
(see Section \ref {sec:Miller}). 

Given a $\mathcal{P}$-indexed filtration $(F^\lambda(X))_{\lambda\in \mathcal{P}}$ of a space $X$, we write $F^{<\lambda}(X):=\bigcup_{\mu<\lambda} F^\mu(X)$.

\begin{corollary} \label{cor:subquotients}
For each $\lambda\in \mathcal{P}$ there is a $U(m)$-equivariant homeomorphism
\[
F^{\lambda} \Hom(\Z^n,U(m)) / F^{<\lambda} \Hom(\Z^n,U(m)) \cong U(m)_+ \wedge_{U(\lambda)}  \bigwedge_{a\in I} C_{|a|}(\mathfrak{u}_{\lambda_a})^+\, .
\]
\end{corollary}
\begin{proof}
The space $F^{<\lambda} \Hom(\Z^n,U(m))$ is a closed subspace of the compact Hausdorff space $F^{\lambda} \Hom(\Z^n,U(m))$, hence
\[
F^{\lambda} \Hom(\Z^n,U(m)) / F^{<\lambda} \Hom(\Z^n,U(m))  \cong \Hom(\Z^n,U(m))_\lambda^+\, .
\]
By Proposition \ref{prop:strata} this is homeomorphic to
\[
\left(U(m) \times_{U(\lambda)}  \prod_{a\in I} C_{|a|}(\mathfrak{u}_{\lambda_a}) \right)^+ \cong U(m)_+ \wedge_{U(\lambda)}  \bigwedge_{a\in I} C_{|a|}(\mathfrak{u}_{\lambda_a})^+\, ,
\]
and these identifications are evidently $U(m)$-equivariant.
\end{proof}

Let $f\co \mathcal{P}\to \N$ be the map of posets sending $\lambda\mapsto \sum_{a\in I} \lambda_a |a|$ and let
\[
F^k\Hom(\Z^n,U(m)) = (f \Psi)^{-1}(\N_{\leq k})
\]
be the induced filtration of $\Hom(\Z^n,U(m))$. For $(g_1,\dots,g_n)\in \Hom(\Z^n,U(m))$, the number $(f\Psi)(g_1,\dots,g_n)$ is the total number of eigenvalues of $g_1,\dots,g_n$ that are different from $1$.

\begin{lemma} \label{lem:cwstructure}
For each $k\geq 0$, the space $F^k\Hom(\Z^n,U(m))$ admits a finite CW-structure in such a way that the inclusion
\[
F^{k-1}\Hom(\Z^n,U(m))\to F^k\Hom(\Z^n,U(m))
\]
is the inclusion of a subcomplex. Moreover, the same statement holds for $U(m)$-equivariant CW-structures instead of ordinary ones.
\end{lemma}
\begin{proof}
The space $\Hom(\Z^n,U(m))$ is evidently an affine real algebraic variety.

An element $(g_1,\dots,g_n)\in \Hom(\Z^n,U(m))$ belongs to $F^k\Hom(\Z^n,U(m))$ if and only if the matrix $\bigoplus_{i=1}^n g_i\in U(mn)$, i.e., the block sum of $g_1,\dots,g_n$, has at least $m-k$ eigenvalues equal to $1$. It follows that $F^k\Hom(\Z^n,U(m))$ is precisely the subset of $\Hom(\Z^n,U(m))$ cut out by the vanishing of all $(m-k+1)\times(m-k+1)$ minors of $\mathds{1}-\bigoplus_{i=1}^n g_i$. These minors are polynomial expressions in the coordinates of the $g_i$, and so this is an algebraic subset of $\Hom(\Z^n,U(m))$. In the same way we see that $F^{k-1}\Hom(\Z^n,U(m))$ is an algebraic subset of $F^k\Hom(\Z^n,U(m))$.

The existence of CW-structures is now guaranteed by the semialgebraic triangulation theorems. In the non-equivariant case we apply \cite[Theorem 3]{Lojasiewicz}. To obtain an equivariant CW-structure we observe that $F^k\Hom(\Z^n,U(m))$ is a semialgebraic $U(m)$-space in the sense of \cite[Definition 3.1]{ParkSuh3}. Then \cite[Theorem 3.5]{ParkSuh3} yields a semialgebraic triangulation of the orbit space $F^k\Hom(\Z^n,U(m))/U(m)$ in such a way that the $U(m)$-isotropy type is constant across each open simplex in the triangulation. Moreover, using the non-equivariant triangulation theorem, we may assume that $F^{k-1}\Hom(\Z^n,U(m))/U(m)$ is a union of open simplices in this triangulation. By passing to the barycentric subdivision these triangulations give compatible equivariant triangulations of the orbit spaces (see \cite[Theorem 5.5]{Illmann}) which then lift to compatible equivariant CW-structures on $F^k\Hom(\Z^n,U(m))$ and $F^{k-1}\Hom(\Z^n,U(m))$ by \cite[Proposition 6.1]{Illmann}.
\end{proof}

In the subposet $f^{-1}(\N_{\leq k})\subseteq \mathcal{P}$ through which $f\Psi$ factors, the points $\lambda\in \mathcal{P}$ with $f(\lambda)=k$ are open. This implies a homeomorphism
\[
F^k\Hom(\Z^n,U(m))\backslash F^{k-1}\Hom(\Z^n,U(m))\cong \bigsqcup_{\lambda\in f^{-1}(k)} \Hom(\Z^n,U(m))_\lambda\, ,
\]
and consequently a homeomorphism
\[
F^k\Hom(\Z^n,U(m))/ F^{k-1}\Hom(\Z^n,U(m))\cong \bigvee_{\lambda \in f^{-1}(k)} \left( U(m)_+ \wedge_{U(\lambda)}  \bigwedge_{a\in I} C_{|a|}(\mathfrak{u}_{\lambda_a})^+\right)\, .
\]
By Lemma \ref{lem:cwstructure} this quotient admits both a non-equivariant as well as a $U(m)$-equivariant CW-structure in which the wedge point is a $0$-cell. These induce CW-structures on each wedge summand.

\begin{corollary} \label{cor:cwstructure}
For every $\lambda\in \mathcal{P}$ the space $U(m)_+ \wedge_{U(\lambda)}  \bigwedge_{a\in I} C_{|a|}(\mathfrak{u}_{\lambda_a})^+$ admits the structure of a finite CW-complex, and also that of a finite $U(m)$-equivariant CW-complex.
\end{corollary}

\section{The stable splitting}\label{sec:splitting}

We now turn to the stable splitting for $\Hom(\Z^n,U(m))$. We begin by briefly explaining the general idea. 

\subsection{Sketch of idea}

It is well-known \cite{Baird} that the action map (\ref{eq:actionmapliegroup}) is an isomorphism in homology with coefficients in $\Z[1/m!]$. After inverting $m!$, this  implies a stable equivalence
\[
 (U(m)/T\times_{\Sigma_m} T^n)_+  \simeq \Hom(\Z^n,U(m))_+\, .
\]
The fat wedge filtration of the torus $T^n_+\cong (S^1)^{mn}_+$ with its permutation action by $\Sigma_{mn}$ admits a $\Sigma_{mn}$-equivariant stable splitting (see e.g. \cite[Proposition 5.2]{CHM}). By restriction to the diagonal subgroup $\Sigma_m\leq \Sigma_{mn}$ one obtains a $\Sigma_m$-equivariant stable splitting of the filtration $(F^\lambda(T^n_+))_{\lambda\in \mathcal{P}}$ defined in Section \ref{sec:construction}. This implies a stable equivalence
\[
(U(m)/T \times_{\Sigma_m} T^n)_+\simeq \bigvee_{\lambda\in \mathcal{P}} (U(m)/T)_+\wedge_{\Sigma_m} F^\lambda(T^n_+)/F^{<\lambda}(T^n_+)\, .
\]
Finally, via the Cayley transform there are maps for all $\lambda\in \mathcal{P}$,
\[
 (U(m)/T)_+\wedge_{\Sigma_m} F^\lambda(T^n_+)/F^{<\lambda}(T^n_+)\to U(m)_+\wedge_{U(\lambda)}  \bigwedge_{a\in I} C_{|a|}(\mathfrak{u}_{\lambda_a})^+\,,
\]
and these are isomorphisms in homology with coefficients in $\Z[1/m!]$ (see Section \ref{sec:commutingvariety}).

\subsection{Statement and corollaries}

What we shall actually prove is the following slightly refined statement from which the above will follow:
\begin{theorem} \label{thm:splitting1}
Let $\lambda\in \mathcal{P}$. After inverting all primes $p\leq \max\{\lambda_a\mid a\in I,\,  |a|> 1\}$ the quotient map
\[
\zeta_\lambda\co F^\lambda\Hom(\Z^n,U(m))_+\to U(m)_+\wedge_{U(\lambda)} \bigwedge_{a\in I} C_{|a|}(\mathfrak{u}_{\lambda_a})^+
\]
collapsing $F^{< \lambda} \Hom(\Z^n,U(m))_+$ admits a stable section up to homotopy.
\end{theorem}

We will prove the theorem in Section \ref{sec:proof}. Let us deduce Theorem \ref{thm:main1} as a corollary:

\begin{corollary} \label{cor:splittingq}
After inverting $m!$ there is a stable equivalence for all integers $n\geq 2$,
\[
\Hom(\Z^n,U(m))_{+}\simeq  \bigvee_{\lambda\in \mathcal{P}} \left(U(m)_+\wedge_{U(\lambda)} \bigwedge_{a\in I} C_{|a|}(\mathfrak{u}_{\lambda_a})^+ \right)\, .
\]
\end{corollary}
\begin{proof}
We consider the filtration of $\Hom(\Z^n,U(m))$ defined in Lemma \ref{lem:cwstructure}. For each $k\leq mn$, the lemma and the subsequent discussion yield a cofibre sequence
\begin{equation*} \label{eq:cofibresequenceq}
F^{k-1}\Hom(\Z^n,U(m))_+ \to F^k\Hom(\Z^n,U(m))_+ \to \bigvee_{\lambda\in f^{-1}(k)} \Hom(\Z^n,U(m))_\lambda^+\, .
\end{equation*}

Now assume that all spaces are localised so as to invert $m!$. By Theorem \ref{thm:splitting1} we can choose, for each $\lambda\in f^{-1}(k)$, a stable section up to homotopy
\[
\sigma_\lambda \co \Hom(\Z^n,U(m))_\lambda^+\to F^\lambda \Hom(\Z^n,U(m))_+
\]
of the projection $\zeta_\lambda$. By taking wedge sum we have a stable map
\[
\bigvee_{\lambda\in f^{-1}(k)} \Hom(\Z^n,U(m))_\lambda^+\xrightarrow{\bigvee_\lambda \sigma_\lambda} \bigvee_{\lambda\in f^{-1}(k)} F^\lambda\Hom(\Z^n,U(m))_{+}\to F^k\Hom(\Z^n,U(m))_+
\]
in which the second map is the sum over the various inclusions. By construction, this is a stable splitting, up to homotopy, of the cofibre sequence above. Induction on $k\leq mn$ yields the asserted equivalence.
\end{proof}

\begin{example} \label{ex:su2}
Consider $\Hom(\Z^2,U(2))$. The stable summands 
that appear in Corollary \ref{cor:splittingq} are listed in 
Table \ref{tbl:homz2u2}. All the stable summands are wedges of spheres, even integrally, except for $C_{2}(\uf_{2})^{+}$. Away from $2$,
\[
C_2(\mathfrak{u}_2)^+\simeq S^4
\]
by Theorem \ref{thm:char0} and the fact that $C_2(\mathfrak{u}_2)^+$ is simply-connected. This may be explained in more geometric terms: First, one has $C_2(\mathfrak{u}_2)^+\cong \Sigma^2 C_2(\mathfrak{su}_2)^+$ by Proposition \ref{prop:centresuspension}. By observing that two elements of $\mathfrak{su}_2$ commute if and only if they are linearly dependent, one sees that $C_2(\mathfrak{su}_2)$ is the quotient of twice the tautological line bundle $L\to \R P^2$ by its zero section. This implies
\[
C_2(\mathfrak{su}_2)^+\cong S(2 L)^{\lozenge}\,,
\]
the unreduced suspension of the sphere bundle in $2 L$. It is explained in \cite[Section 4]{Crabb} that this space is stably equivalent to $S^2\vee \Sigma^2\R P^2$, and so $C_2(\mathfrak{su}_2)^+\simeq S^2$ away from $2$. As a result, there is a stable equivalence away from $2$
\[
\Hom(\Z^2,U(2))_+\simeq S^{0}\vee (\bigvee^{2} S^{1})
\vee (\bigvee^{2} S^{2})\vee (\bigvee^{4} S^{3})\vee (\bigvee^{5} S^{4})\vee (\bigvee^{2} S^{5})\, .
\]

\begin{table}[h]
\def\arraystretch{1.5}
\begin{tabular}{|l|l|l|l|l|}
\hline
(0,0) & (0,1) & (1,0) & (1,1)  & Stable summand \\ \hline\hline
 $2$ & $0$ & $0$ & $0$   & $S^0$  \\ \hline
$0$ & $2$ & $0$ & $0$   & $C_{1}(\uf_{2})^{+}\cong S^{4}$ \\ \hline
$0$ & $0$ & $2$ & $0$   & $C_{1}(\uf_{2})^{+}\cong  S^{4}$ \\ \hline
 $0$ & $0$ & $0$ & $2$   & $C_{2}(\uf_{2})^{+}$ \\ \hline
$1$ & $1$ & $0$ & $0$   & $(U(2)/U(1)\times_{U(1)}C_{1}(\uf_{1}))^{+}\cong S^{1}\vee S^{3}$ \\ \hline
$1$ & $0$ & $1$ & $0$   & $(U(2)/U(1)\times_{U(1)}C_{1}(\uf_{1}))^{+}\cong  S^{1}\vee S^{3}$ \\ \hline
$1$ & $0$ & $0$ & $1$   & $(U(2)/U(1)\times_{U(1)}C_{2}(\uf_{1}))^{+}\cong  S^{2}\vee S^{4}$ \\ \hline
$0$ & $1$ & $1$ & $0$   & $(U(2)/U(1)\times_{U(1)}C_{2}(\uf_{1}))^{+}\cong S^{2}\vee S^{4}$ \\ \hline
$0$ & $1$ & $0$ & $1$   & $(U(2)\times_{U(1)\times U(1)}C_{1}(\uf_{1})\times C_{2}(\uf_{1}))^{+}\cong S^{3}\vee S^{5}$ \\ \hline
$0$ & $0$ & $1$ & $1$  & $(U(2)\times_{U(1)\times U(1)}C_{1}(\uf_{1})\times C_{2}(\uf_{1}))^{+}\cong  S^{3}\vee S^{5}$ \\ \hline
\end{tabular}
\caption{The table shows the stable summands of $\Hom(\Z^2,U(2))_+$. The partitions are indexed by $\{0,1\}^2$ and their parts are listed in the corresponding columns. } \label{tbl:homz2u2}
\end{table}
\end{example}

We can use Theorem \ref{thm:splitting1} to identify the summands in the stable splitting of $\Hom(\Z^n,U(m))$ proved by Adem and Cohen \cite[Theorem 6.5]{AC}. These are of the form
\[
\Hom(\Z^n,U(m))/S_n(U(m))
\]
where $S_n(U(m))\subseteq \Hom(\Z^n,U(m))$ is the subspace for which at least one coordinate is the identity.

Let $\mathcal{S}\subseteq \mathcal{P}$ be the subset defined by
\[
\mathcal{S}=\{\lambda\in \mathcal{P} \mid \sum_{a\in I} \lambda_a a(i)\neq 0\textnormal{ for all }1\leq i \leq n  \}\, .
\]
Observe that if $\lambda\in \mathcal{S}$, then $\Hom(\Z^n,U(m))_\lambda\subseteq \Hom(\Z^n,U(m))\backslash S_n(U(m))$.

\begin{corollary} \label{cor:summandsinAC}
After inverting $m!$ there is a stable equivalence
\[
\Hom(\Z^n,U(m))/S_n(U(m))\simeq \bigvee_{\lambda\in \mathcal{S}} \left(U(m)_+\wedge_{U(\lambda)} \bigwedge_{a\in I} C_{|a|}(\mathfrak{u}_{\lambda_a})^+\right)\, .
\]
\end{corollary}
\begin{proof}
By definition of $\mathcal{S}$ we have that
\[
S_n(U(m))=\bigcup_{\lambda\in \mathcal{P}\backslash \mathcal{S} } F^\lambda\Hom(\Z^n,U(m))\, .
\]
Let $F^k\Hom(\Z^n,U(m))$ denote the filtration used in the proof of Corollary \ref{cor:splittingq}. The induced filtration of the subspace $S_n(U(m))$ satisfies
\[
F^kS_n(U(m))=\bigcup_{\lambda\in f^{-1}(\N_{\leq k})\cap (\mathcal{P}\backslash \mathcal{S})} F^\lambda \Hom(\Z^n,U(m))\, .
\]
An argument analogous to that of Lemma \ref{lem:cwstructure} shows that the inclusions $F^{k-1}S_n(U(m))\to F^kS_n(U(m))$ are inclusions of subcomplexes in a CW-structure. Thus, there is a cofibre sequence of the form
\[
F^{k-1}S_n(U(m))_+ \to F^kS_n(U(m))_+\to \bigvee_{\lambda\in f^{-1}(k)\cap (\mathcal{P}\backslash \mathcal{S})}  \Hom(\Z^n,U(m))_{\lambda}^+\, .
\]

The maps  $F^kS_n(U(m))\to F^k\Hom(\Z^n,U(m))$  are likewise inclusions of subcomplexes. Therefore, the space $\Hom(\Z^n,U(m))/S_n(U(m))$ is filtered by the subspaces
\[
\bar{F}^k:=F^k\Hom(\Z^n,U(m))/F^kS_n(U(m))
\]
which are in an induced cofibre sequence
\[
\bar{F}^{k-1} \to \bar{F}^k \to \bigvee_{\lambda\in f^{-1}(k) \cap \mathcal{S}}  \Hom(\Z^n,U(m))_{\lambda}^+\, .
\]

The proof is finished by showing that, after inverting $m!$, this cofibre sequence splits for every $k\leq mn$. This is done in the same fashion as in the proof of Corollary \ref{cor:splittingq}.
\end{proof}

\subsection{Proof of Theorem \ref{thm:splitting1}} \label{sec:proof}  We begin by constructing a suitable approximation of the stable summands $U(m)_+\wedge_{U(\lambda)} \bigwedge_{a\in I} C_{|a|}(\mathfrak{u}_{\lambda_a})^+$.

For the unitary group $U(k)$ we choose a maximal torus $T$ and denote by $\mathfrak{t}\subseteq \mathfrak{u}_k$ its Lie algebra. Then $\mathfrak{t}$ is naturally a $\Sigma_k$-representation and we can form the vector bundle $U(k)/T\times_{\Sigma_k} \mathfrak{t}^n$ associated with the representation $\mathfrak{t}^n=\mathfrak{t}\oplus \cdots \oplus \mathfrak{t}$. As we show in Lemma \ref{lem:actionmapliealgebra}, the adjoint representation can be used to define a map from the Thom space
\[
\phi^+ \co (U(k)/T)_+\wedge_{\Sigma_k} (\mathfrak{t}^+)^{\wedge n}\to C_n(\mathfrak{u}_k)^+
\]
which is a homology isomorphism with coefficients in $\Z[1/p\mid p\leq k]\subseteq \Q$.

Now fix $\lambda\in \mathcal{P}$. To keep track of the rank we write $T(\lambda_a)$ for the maximal torus of $U(\lambda_a)$ and $\mathfrak{t}_{\lambda_a}$ for its Lie algebra. For every $a\in I$ we define a based $U(\lambda_a)$-space $\mathfrak{X}_a$ by
\[
\mathfrak{X}_a=\begin{cases} S^0 & \textnormal{if } |a|=0,\\ \mathfrak{u}_{\lambda_a}^+ & \textnormal{if } |a|=1, \\  (U(\lambda_a)/T(\lambda_a))_+\wedge_{\Sigma_{\lambda_a}} (\mathfrak{t}_{\lambda_a}^+)^{\wedge |a|} & \textnormal{if } |a|>1. \end{cases}
\]
The $U(\lambda_a)$-action is the adjoint action when $|a|=1$ and the obvious action on $U(\lambda_a)/T(\lambda_a)$ when $|a|>1$.
The construction above supplies us with based $U(\lambda_a)$-equivariant maps
\[
\phi_a^+\co \mathfrak{X}_a\to C_{|a|}(\mathfrak{u}_{\lambda_a})^+
\]
for all $a\in I$ when $|a|>1$. When $|a|=0,1$ we let $\phi_a^+$ be the identity map. 

Let $Q(\lambda)$ be the set of primes $p$ such that $p\leq \max\{\lambda_a\mid a\in I,\,  |a|> 1\}$.
\begin{lemma} \label{lem:approximation}
After inverting all primes $p\in Q(\lambda)$, the map
\[
id \wedge \bigwedge_{a\in I} \phi_a^+\co U(m)_+\wedge_{U(\lambda)} \bigwedge_{a\in I} \mathfrak{X}_a \to U(m)_+\wedge_{U(\lambda)} \bigwedge_{a\in I} C_{|a|}(\mathfrak{u}_{\lambda_a})^+
\]
is a stable homotopy equivalence.
\end{lemma}
\begin{proof}
Let $R$ denote the subring $\Z[1/p\mid p\in Q(\lambda)]\subseteq \Q$. The space $U(m)_+\wedge_{U(\lambda)} \bigwedge_{a\in I} C_{|a|}(\mathfrak{u}_{\lambda_a})^+$ admits a CW-structure by Corollary \ref{cor:cwstructure}. By Whitehead's theorem it therefore suffices to show that $id \wedge \bigwedge_{a\in I} \phi_a^+$ is a homology isomorphism with coefficients in $R$.

By Lemma \ref{lem:actionmapliealgebra} every $\phi_a^+$, $a\in I$ induces an isomorphism on homology with $R$-coefficients. Then $\bigwedge_{a\in I} \phi_a^+$ is a homology isomorphism by the K{\"u}nneth theorem, and
\[
id\times \bigwedge_{a\in I} \phi_a^+\co U(m) \times_{U(\lambda)} \bigwedge_{a\in I} \mathfrak{X}_a \to U(m)\times_{U(\lambda)} \bigwedge_{a\in I} C_{|a|}(\mathfrak{u}_{\lambda_a})^+
\]
is a homology isomorphism by comparison of Serre spectral sequences. Both of these bundles have canonical sections at infinity. Collapsing these and using a five lemma argument we see that $id \wedge \bigwedge_{a\in I} \phi_a^+$ is a homology isomorphism with $R$-coefficients as well.
\end{proof}

The next step will be the passage from the Lie algebra to the Lie group level.

For the maximal torus $T=T(k)$ of $U(k)$ let $F^{j}(T)\subseteq T$ be the $j$-th stage of the fat wedge filtration. Then
\[
T\backslash F^{k-1}(T)=\{(t_1,\dots,t_k)\in T \mid t_i\neq 1 \textnormal{ for all }1\leq i \leq k\}\, .
\]

The Cayley transform of Definition \ref{def:cayley} restricts to a $\Sigma_k$-equivariant homeomorphism $\mathfrak{t}\cong T\backslash F^{k-1}(T)$, and this extends to a $\Sigma_k$-equivariant homeomorphism
\[
\mathfrak{t}^+\cong T/F^{k-1}(T)\,.
\]
This identification yields a $\Sigma_k$-equivariant quotient map $c_T \co T_+\to \mathfrak{t}^+$. Its $n$-th smash power is the $\Sigma_k$-equivariant map
\[
c_T^{\wedge n} \co T^n_+\cong (T_+)^{\wedge n} \to (\mathfrak{t}^+)^{\wedge n}
\]
that collapses the subspace $F^{nk-1}(T^n)\subseteq T^n$.

The fat wedge filtration of $T^n\cong (S^1)^{nk}$ splits $\Sigma_{nk}$-equivariantly by \cite[Proposition 5.2]{CHM}. In particular, there is a $\Sigma_{k}$-equivariant stable section up to homotopy of $c_T^{\wedge n}$. More precisely, let $V=\R^{nk}$ be the standard permutation representation of $\Sigma_{nk}$, and write $V$ also for its restriction to a $\Sigma_k$-representation along the diagonal inclusion $\Sigma_k\leq \Sigma_{nk}$. Then there is a $\Sigma_k$-equivariant map
\begin{equation} \label{eq:torustopcell}
s_{T^n}\co V^+\wedge (\mathfrak{t}^+)^{\wedge n} \to V^+\wedge T^n_+
\end{equation}
and a $\Sigma_k$-equivariant homotopy $(id\wedge c_{T}^{\wedge n})s_{T^n}\simeq id$.

For every $a\in I$ we define a based $U(\lambda_a)$-space $X_a$ by
\[
X_a=\begin{cases} S^0 & \textnormal{if } |a|=0,\\ U(\lambda_a)_+ & \textnormal{if } |a|=1, \\  (U(\lambda_a)/T(\lambda_a))_+\wedge_{\Sigma_{\lambda_a}} (T(\lambda_a)_+)^{\wedge |a|} & \textnormal{if } |a|>1. \end{cases}
\]
The $U(\lambda_a)$-action in the case $|a|=1$ is by conjugation. There are based $U(\lambda_a)$-equivariant maps, for all $a\in I$
\[
\pi_a\co X_a\to \mathfrak{X}_a
\]
defined as follows: When $|a|>1$ we let $\pi_a=id \wedge c_T^{\wedge |a|}$, and when $|a|=0,1$ we let $\pi_a$ be the identity map, respectively the projection $U(\lambda_a)_+\to \mathfrak{u}_{\lambda_a}^+$ onto the top filtration quotient of $U(\lambda_a)$ described in (\ref{eq:millertopstratum}).

\begin{lemma}\label{lem:presection}
The map
\[
id\wedge \bigwedge_{a\in I} \pi_a\co U(m)_+\wedge_{U(\lambda)} \bigwedge_{a\in I} X_a \to U(m)_+\wedge_{U(\lambda)} \bigwedge_{a\in I} \mathfrak{X}_a
\]
has a stable section up to homotopy.
\end{lemma}

For the proof we require the following lemma.

\begin{lemma} \label{lem:stablesection}
Let $G$ be a compact Lie group and $H\subseteq G$ a closed subgroup. Let $X,Y$ be based $H$-spaces and $f\co X\to Y$ a based $H$-equivariant map. If $f$ has an $H$-equivariant stable section up to homotopy, then $id\wedge f\co G_+\wedge_{H}X\to G_+\wedge_H Y$ has a $G$-equivariant stable section up to homotopy.
\end{lemma}
\begin{proof}
By assumption there is a finite dimensional $H$-representation $V$, an $H$-equivariant map
\[
s\co V^+\wedge Y\to V^+\wedge X
\]
and an $H$-equivariant homotopy $(id\wedge f) s\simeq id$.

Consider the $G$-vector bundle
\[
E=G\times_H V\to G/H\, .
\]
If $p\co G\to G/H$ is the projection, then the pullback
\[
p^\ast E=\{(g_1,[g_2,v])\in G\times E\mid g_1H=g_2H\}
\]
is a $(G\times H)$-vector bundle with action defined by
\[
(g,h)\cdot (g_1,[g_2,v])=(gg_1h^{-1},[gg_1,v])\, .
\]

There is a $(G\times H)$-equivariant bundle isomorphism
\[
p^\ast E\cong G\times V
\]
defined by $(g_1,[g_2,v])\mapsto (g_1,g_1^{-1}g_2 v)$ with inverse $(g,v)\mapsto (g,[g,v])$. The action on $G\times V$ is given by $(g,h)\cdot (g_1,v)=(gg_1h^{-1},hv)$.

Since $G$ is compact, we can find a $G$-vector bundle $F \to G/H$ such that $F\oplus E$ is a trivial $G$-vector bundle, i.e.,
\[
F\oplus E\cong G/H \times W
\]
for some $G$-representation $W$. There is then an isomorphism of $(G\times H)$-vector bundles
\[
G\times W\cong p^\ast(F\oplus E)\cong p^\ast F \times  V\, ,
\]
where $V$ is the trivial $G$-bundle over a point with $H$ acting trivially on it. The $(G\times H)$-action on $G\times W$ is given by $(g,h)\cdot (g_1,w)=(gg_1h^{-1},gw)$.

Taking Thom spaces on both sides gives a $(G\times H)$-equivariant homeomorphism
\[
G_+\wedge W^+ \cong Z\wedge V^+\,,
\]
where $Z$ is the Thom space of $p^\ast F$.

Now we consider the composite
\[
W^+\wedge G_+\wedge Y \cong Z\wedge V^+\wedge Y\xrightarrow{id\wedge s} Z\wedge V^+\wedge X \cong W^+\wedge G_+\wedge X\, .
\]
By construction this is $(G\times H)$-equivariant and the composite with $id\wedge f$ is $(G\times H)$-equivarianty homotopic to the identity. By passing to the $H$-quotient we obtain the desired $G$-equivariant stable section up to homotopy.
\end{proof}

\begin{proof}[Proof of Lemma \ref{lem:presection}]
By Lemma \ref{lem:stablesection} it suffices to show that for every $a\in I$ the map $\pi_a\co X_a\to \mathfrak{X}_a$ has a $U(\lambda_a)$-equivariant stable section $\sigma_a\co \mathfrak{X}_a\to X_a$ up to homotopy.

For $|a|=0$ this is trivial, and for $|a|=1$ a stable section is provided by a Pontryagin-Thom collapse (see Section \ref{sec:Miller}). Now assume $|a|>1$. We can view $\pi_a$ as the map
\[
id\wedge c_T^{\wedge |a|}\co U(\lambda_a)_+\wedge_{N(T(\lambda_a))} (T(\lambda_a)_+)^{\wedge |a|}\to U(\lambda_a)_+\wedge_{N(T(\lambda_a))} (\mathfrak{t}_{\lambda_a}^+)^{\wedge |a|}
\]
where $N(T(\lambda_a))$ is the normaliser of $T(\lambda_a)$ in $U(\lambda_a)$. Once again, by Lemma \ref{lem:stablesection}, it suffices to give a $N(T(\lambda_a))$-equivariant stable section up to homotopy of $c_T^{\wedge |a|}$. But such a section is given by (\ref{eq:torustopcell}).
\end{proof}

Finally, we will construct a map
\[
\varphi_\lambda\co U(m)_+\wedge_{U(\lambda)} \bigwedge_{a\in I} X_a \to F^\lambda\Hom(\Z^n,U(m))_+\, .
\]
First we define, for every $a\in I$, a based map 
\[
\varphi_a^+\co X_a\to \Hom(\Z^{|a|},U(\lambda_a))_+\, .
\]
For $|a|>1$ we let $\varphi_a^+$ be the action map defined in (\ref{eq:actionmapliegroup}), extended to a based map over the disjoint basepoint. For $|a|=0,1$ we choose $\varphi_a^+$ to be the identity map.

Let $\beta_\lambda$ be the composite
\[
\bigwedge_{a\in I} X_a \xrightarrow{\bigwedge_{a\in I} \textnormal{sh}_a^+ \varphi_{a}^+} \bigwedge_{a\in I} \Hom(\Z^n,U(\lambda_a))_+ \xrightarrow{\oplus} \Hom(\Z^n,U(m))_+
\]
where $\textnormal{sh}_a$ is the map defined in (\ref{eq:shufflemap}) and the last map takes the block sum of matrices. By construction, the image of $\beta_\lambda$ is in the subspace $F^\lambda \Hom(\Z^n,U(m))_+$. We then define $\varphi_\lambda$ as the composite
\[
U(m)_+\wedge_{U(\lambda)} \bigwedge_{a\in I} X_a \xrightarrow{id\wedge \beta_\lambda} U(m)_+\wedge_{U(\lambda)} F^\lambda\Hom(\Z^n,U(m))_+ \to F^\lambda \Hom(\Z^n,U(m))_+\,,
\]
in which the second map is conjugation. This is well-defined, because $F^\lambda \Hom(\Z^n,U(m))_+$ is invariant under the conjugation action by $U(m)$.

To finish the proof of Theorem \ref{thm:splitting1} we consider the diagram
\[
\xymatrix{
U(m)_+\wedge_{U(\lambda)} \bigwedge_{a\in I} X_a  \ar[r]^-{id\wedge \bigwedge_{a\in I} \pi_a} \ar[d]^-{\varphi_\lambda} & U(m)_+\wedge_{U(\lambda)} \bigwedge_{a\in I} \mathfrak{X}_a \ar[d]_-{\simeq}^-{id\wedge \bigwedge_{a\in I} \phi_a^+} \\
F^\lambda\Hom(\Z^n,U(m))_+ \ar[r]^-{\zeta_\lambda} & U(m)_+\wedge_{U(\lambda)} \bigwedge_{a\in I} C_{|a|}(\mathfrak{u}_{\lambda_a})^+\, .
}
\]
A simple diagram chase shows that the diagram commutes. Let $\alpha$ be a stable homotopy inverse of the right hand vertical map, and let $\tilde{\sigma}_{\lambda}$ be a stable section, up to homotopy, of the top horizontal map. Then the composite $\varphi_\lambda\tilde{\sigma}_\lambda\alpha$ is a stable section, up to homotopy, of $\zeta_\lambda$.

\subsection{Symplectic and orthogonal groups} \label{sec:symplectic} Stable decompositions for spaces of commuting elements in the symplectic and orthogonal groups can be obtained with a proof that is almost identical with the one in the unitary case.

The Weyl group of $Sp(m)$ and of $SO(2m+1)$ is the signed symmetric group, i.e., the wreath product $\Z/2\wr \Sigma_m$. It acts on a maximal torus $T\cong (S^1)^m$ by permutation and inversion of the factors. The Weyl group of $SO(2m)$ is the subgroup of $\Z/2\wr \Sigma_m$ consisting of those signed permutations with an even number of minus signs.

Now let $W$ denote any of these Weyl groups. As before, we denote by $\mathcal{P}$ the poset of partitions of $m$ indexed by $I=\{0,1\}^n$. The $\mathcal{P}$-indexed filtration of $T^n$ constructed in Section \ref{sec:construction} and represented by the map $\Psi_T\co T^n\to \mathcal{P}$ is invariant under the $W$-action and thus it induces a map $\bar{\Psi}_T\co T^n/W\to \mathcal{P}$.

Let us focus first on the symplectic groups. Every abelian subgroup of $Sp(m)$ is contained in a maximal torus. This implies that $\Hom(\Z^n,Sp(m))$ is path-connected, and we can identify the space $\Hom(\Z^n,Sp(m))/Sp(m)$ with $T^n/(\Z/2\wr \Sigma_m)$. Just like in the unitary case (see Construction \ref{cons:construction}) we obtain an induced filtration
\[
\Psi_{Sp(m)}\co \Hom(\Z^n,Sp(m))\to \mathcal{P}\, .
\]

The next step is the identification of the strata with bundles of commuting varieties. This works again by use of the Cayley transform. The Cayley transform $X\mapsto (X-\mathds{1})(X+\mathds{1})^{-1}$ is defined on the Lie algebra of $Sp(m)$,
\[
\mathfrak{sp}_m=\{X\in \textnormal{Mat}_m(\mathbb{H})\mid X^\dagger=-X\}\, .
\]
As in the complex case, it is an equivariant diffeomorphism of $\mathfrak{sp}_m$ with the open subspace of $Sp(m)$ consisting of those matrices $g$ such that $g-\mathds{1}$ is invertible. This leads to the obvious analogue of Corollary \ref{cor:subquotients} in which each unitary group is replaced by the corresponding symplectic group.

To show that the filtration splits stably, we proceed as in the unitary case. However, this requires a stable splitting of the map $c_{T}^{\wedge n}\co T^n_+\to (\mathfrak{t}^+)^{\wedge n}$ which is equivariant for the action of the wreath product $\Z/2\wr \Sigma_m$. The stable splitting proved in \cite[Proposition 5.2]{CHM} does have this property, see \cite[Remark 5.3]{CHM}. We arrive at the following theorem:
\begin{theorem} \label{thm:splittingspm}
After inverting $2^m m!$ there is a stable splitting for all $n\in \N$,
\[
\Hom(\Z^n,Sp(m))_+\simeq \bigvee_{\lambda\in \mathcal{P}} \left(Sp(m)_+\wedge_{\prod_{a\in I}Sp(\lambda_a)} \bigwedge_{a\in I} C_{|a|}(\mathfrak{sp}_{\lambda_a})^+\right)\, .
\]
\end{theorem}

\begin{example} \label{ex:sp1}
Consider $Sp(1)=SU(2)$. In this case $\mathcal{P}$ is in bijection with $I=\{0,1\}^n$, by sending $\lambda=(\lambda_a)_{a\in I}$ to the unique $a$ such that $\lambda_a=1$. The resulting decomposition away from $2$ takes the form
\[
\Hom(\Z^n,SU(2))_+\simeq \bigvee_{0\leq l\leq n} \left(\bigvee^{\binom{n}{l}} C_l(\mathfrak{su}_2)^+\right)\, .
\]
In fact, the Cayley transform $SU(2)\backslash \{\mathds{1}\}\cong \mathfrak{su}_2$ yields an ``accidental'' homeomorphism
\[
C_l(\mathfrak{su}_2)^+\cong \Hom(\Z^l,SU(2))/S_l(SU(2))
\]
which shows that the splitting holds even without inverting $2$, see \cite[Theorem 6.5]{AC}.
\end{example}

Now we consider the orthogonal groups. Let $k\in \{2m,2m+1\}$. When $k\geq 3$ some extra care must be taken, because if in addition $n\geq 2$, then $\Hom(\Z^n,SO(k))$ is no longer path-connected. In particular, the action map (defined as in (\ref{eq:actionmapliegroup}))
\[
\varphi\co SO(k)/T\times_W T^n\to \Hom(\Z^n,SO(k))
\]
maps only onto the path-component of the trivial homomorphism. Let $\Hom(\Z^n,SO(k))_{\mathds{1}}$ denote this path-component. Then $\Hom(\Z^n,SO(k))_{\mathds{1}}/SO(k)\cong T^n/W$ and $\bar{\Psi}_T$ induces a filtration
\[
\Psi_{SO(k)}\co \Hom(\Z^n,SO(k))_{\mathds{1}}\to \mathcal{P}\,.
\]
Observe that $\Hom(\Z^n,O(k))_{\mathds{1}}=\Hom(\Z^n,SO(k))_{\mathds{1}}$, which is why we need not consider the full orthogonal group.

A characterisation of the path-components of $\Hom(\Z^n,SO(k))$ was obtained in \cite{Rojo}. In terms of representation theory, a representation $\Z^n\to SO(k)$ belongs to the path-component $\Hom(\Z^n,SO(k))_{\mathds{1}}$ if and only if it is isomorphic to a direct sum of $m$ representations $\Z^n\to SO(2)$, and an additional one-dimensional trivial representation if $k=2m+1$.

Using the real Cayley transform $\psi_{2m}\co \mathfrak{so}_{2m}\to SO(2m)$, defined by the usual formula $X\mapsto (X-\mathds{1})(X+\mathds{1})^{-1}$, we arrive at the following analogue of Corollary \ref{cor:subquotients}:  There is an equivariant homeomorphism
\[
\Psi_{SO(k)}^{-1}(\lambda)^+\cong (O(k)/O(2\lambda_0+\varepsilon))_+\wedge_{\prod_{0\neq a\in I}O(2\lambda_a)} \bigwedge_{a\in I} C_{|a|}(\mathfrak{so}_{2\lambda_a})^+\,,
\]
where $\varepsilon=0$ if $k$ is even and $\varepsilon=1$ if $k$ is odd.

Proceeding as before we obtain:

\begin{theorem} \label{thm:splittingsom}
Let $k\in \{2m,2m+1\}$ and let $\varepsilon=0$ if $k$ is even and $\varepsilon=1$ if $k$ is odd. After inverting $2^{m-1+\varepsilon} m!$ there is a stable splitting for all $n\in \N$,
\[
\Hom(\Z^n,SO(k))^+_{\mathds{1}}\simeq \bigvee_{\lambda\in \mathcal{P}}\left((O(k)/O(2\lambda_0+\varepsilon))_+\wedge_{\prod_{0\neq a\in I}O(2\lambda_a)} \bigwedge_{a\in I} C_{|a|}(\mathfrak{so}_{2\lambda_a})^+\right)\, .
\]
\end{theorem}

\begin{example}
Consider the case $k=3$. Then $m=1$ and thus, as in Example \ref{ex:sp1}, the poset of partitions is in bijection with $I=\{0,1\}^n$. The theorem implies a stable decomposition
\[
\Hom(\Z^n,SO(3))_{\mathds{1}}^+ \simeq S^0\vee  \bigvee_{1\leq l \leq n}\left( \bigvee^{\binom{n}{l}} (O(3)/O(1))_+\wedge_{O(2)} C_l(\mathfrak{so}_2)^+\right)\, .
\]
The summand $S^0$ comes from the partition $\lambda\in \mathcal{P}$ with $\lambda_0=1$. A priori the splitting holds away from $2$. However, the only commuting varieties that appear are those for $\mathfrak{so}_2$. Because the Weyl group is trivial for $\mathfrak{so}_2$, inspection of the proof then shows that the splitting holds without inverting $2$ (namely, the step of Lemma \ref{lem:approximation} can be achieved without inverting primes). To identify the summands we note that $C_l(\mathfrak{so}_2)=\mathfrak{so}_2^l$ and that the subgroup $SO(2)\subset O(2)$ acts trivially on $\mathfrak{so}_2^l$. The residual action of $O(2)/SO(2)\cong \Z/2$ on $\mathfrak{so}_2\cong \R$ is via sign, so that
\[
O(3)/O(1)\times_{O(2)} \mathfrak{so}_2^l\cong L\oplus \cdots \oplus L \; \textnormal{($l$ times)}\,,
\]
where $L$ is the tautological line bundle over $\R P^2$. The Thom space of $lL$ is known to be homeomorphic with the stunted projective space $\R P^{l+2}/\R P^{l-1}$ (see e.g. \cite[Lemma 3.2]{BJS}). Therefore, we obtain a stable splitting
\[
\Hom(\Z^n,SO(3))_\mathds{1} \simeq \bigvee_{1\leq l \leq n} \left(\bigvee^{\binom{n}{l}} \R P^{l+2}/\R P^{l-1}\right)\, .
\]
(Here we have omitted the disjoint basepoint.) This splitting agrees with the one obtained in \cite[Section 7]{ACG} and with the homology computations given in \cite{Sjerve}. 
\end{example}

\begin{remark}
Notice that when we choose $n=1$ in the previous example, the stable splitting sees no interesting information as $SO(3)\cong \R P^3$. In contrast, Miller's splitting of $SO(n)$ produces the well-known stable equivalence $SO(3)\simeq \R P^2\vee S^3$.
\end{remark}

\section{Steenrod powers and a negative result} \label{sec:nonsplitting}

In this section we show that the stable splitting of Theorem \ref{thm:splitting1} does not hold integrally:

\begin{theorem} \label{thm:nonsplitting}
Let $p$ be a prime and let $n\geq 2$ be an integer. The quotient map
\[
\zeta\co \Hom(\Z^n,U(p))_+\to C_n(\mathfrak{u}_{p})^+
\]
does not have a stable section up to homotopy at the prime $p$.
\end{theorem}

If $X$ is a space and $n\in \mathbb{N}$, we denote by
\[
\textnormal{SP}^n(X):=X^n/\Sigma_n
\]
the $n$-th symmetric power of $X$. If $X$ is based, there is an evident inclusion map $\textnormal{SP}^{n-1}(X)\to \textnormal{SP}^n(X)$ whose cofibre is the reduced symmetric power
\[
\overline{\textnormal{SP}}^n(X):= \textnormal{SP}^n(X)/\textnormal{SP}^{n-1}(X)\cong X^{\wedge n}/\Sigma_n\, .
\]
We will prove Theorem \ref{thm:nonsplitting} by comparing $\Hom(\Z^n,U(p))$ and $C_n(\mathfrak{u}_p)^+$ to the $p$-th (reduced) symmetric powers of $S^n$. Steenrod operations give obstructions for a stable splitting of the natural map $\textnormal{SP}^p(S^n)\to \overline{\textnormal{SP}}^p(S^n)$. We will lift those obstructions to $\Hom(\Z^n,U(p))$.

Fix $s\geq 1$. We will carry out most of the proof of Theorem \ref{thm:nonsplitting} for the unitary group of rank $p^s$; only in Lemma \ref{lem:maptosphere} will we have to choose $s=1$.

Let $T\subseteq U(p^s)$ be a maximal torus. Observe that every commuting $n$-tuple in $U(p^s)$ is conjugate to an $n$-tuple in $T$. It follows that the quotient of $\Hom(\Z^n,U(p^s))$ by the adjoint action of $U(p^s)$ is homeomorphic to the quotient of $T^n$ by the diagonal action of the Weyl group $\Sigma_{p^s}$. Since $T\cong (S^1)^{p^s}$ and $\Sigma_{p^s}$ permutes the factors, we obtain a homeomorphism
\[
\Hom(\Z^n,U(p^s))/U(p^s)\cong \textnormal{SP}^{p^s}((S^1)^n)\,.
\]
Collapsing further the $(n-1)$-skeleton of $(S^1)^n$ we obtain a composite map
\[
\pi_n\co \Hom(\Z^n,U(p^s))\to \textnormal{SP}^{p^s}(S^n)\,.
\]

The discussion for the commuting variety $C_n(\mathfrak{u}_{p^s})^+$ is similar. Let $\mathfrak{t}\subseteq \mathfrak{u}_{p^s}$ be the Lie algebra of $T$. As every commuting $n$-tuple in $\mathfrak{u}_{p^s}$ is contained in a Cartan subalgebra of $\mathfrak{u}_{p^s}$ (see Proposition \ref{prop:cartansubalgebra} in the appendix) we can identify $C_n(\mathfrak{u}_p)^+/U(p)$ with the quotient of $(\mathfrak{t}^n)^+$ by the diagonal action of $\Sigma_{p^s}$.  As $(\mathfrak{t}^n)^+\cong ((S^1)^{\wedge n})^{\wedge p^s}$ with $\Sigma_{p^s}$ permuting the smash factors, there is a homeomorphism
\[
C_n(\mathfrak{u}_p)^+/U(p)\cong \overline{\textnormal{SP}}^{p^s}(S^n)\,.
\]
In particular, we have a quotient map
\[
\bar{\pi}_n\co C_n(\mathfrak{u}_{p^s})^+\to \overline{\textnormal{SP}}^{p^s}(S^n)\,.
\]

It is easily seen that the maps $\pi_n$, $\bar{\pi}_n$ and the canonical projection $\bar{\zeta}\co \textnormal{SP}^{p^s}(S^n)\to  \overline{\textnormal{SP}}^{p^s}(S^n)$ fit into a commutative diagram
\begin{equation} \label{dgr:symmetricproduct}
\xymatrix{
\Hom(\Z^n,U(p^s)) \ar[r]^-{\pi_n} \ar[d]^-{\zeta} & \textnormal{SP}^{p^s}(S^n) \ar[d]^-{\bar{\zeta}} \\
C_n(\mathfrak{u}_{p^s})^+ \ar[r]^-{\bar{\pi}_n} & \overline{\textnormal{SP}}^{p^s}(S^n) .
}
\end{equation}

Let $\mathscr{P}^i$ denote the Steenrod reduced $p$-th power; when $p=2$ we interpret $\mathscr{P}^i$ as $\textnormal{Sq}^{2i}$. In what follows we will omit the coefficient ring for cohomology; if not noted otherwise, it will be $\F_p$. 

The mod-$p$ cohomology of $\textnormal{SP}^{p^s}(S^n)$ was computed by Nakaoka \cite{Nakaoka}. Let $u_n\in H^n(\textnormal{SP}^{p^s}(S^n))\cong \F_p$ be a generator. By \cite[Corollary 2]{Nakaoka} there exists a unique class
\[
v_n\in H^{n+2(p^s-1)}(\overline{\textnormal{SP}}^{p^s}(S^n))
\]
such that $\bar{\zeta}^\ast(v_n)=(\mathscr{P}^{p^{s-1}} \cdots  \mathscr{P}^1)(u_n)$. Setting $a_{n,s}:=\pi_n^\ast(u_n)$ and $b_{n,s}:=\bar{\pi}_n^\ast(v_n)$ we have that
\[
\zeta^\ast(b_{n,s})=(\mathscr{P}^{p^{s-1}} \cdots  \mathscr{P}^1)(a_{n,s})\, .
\]
The class $b_{n,s}$ will be the obstruction used to prove Theorem \ref{thm:nonsplitting}:
\begin{lemma} \label{lem:nonsplitting}
Suppose that $b_{n,s}\neq 0$. Then the quotient map
\[
\zeta\co \Hom(\Z^n,U(p^s))_+\to  C_n(\mathfrak{u}_{p^s})^+
\]
does not have a stable section up to homotopy at the prime $p$.
\end{lemma}
\begin{proof}
Suppose that, at the prime $p$, there is a \emph{stable} map
\[
\sigma\co  C_n(\mathfrak{u}_{p^s})^+\to \Hom(\Z^n,U(p^s))_+
\]
such that $\zeta\sigma\simeq id$. Stability of the Steenrod powers gives 
\[
0\neq b_{n,s}=(\zeta\sigma)^\ast(b_{n,s})=(\mathscr{P}^{p^{s-1}} \cdots  \mathscr{P}^1)(\sigma^\ast(a_{n,s}))\, .
\]
By Proposition \ref{prop:centresuspension} in the appendix, $C_n(\mathfrak{u}_{p^s})^+$ is at least $n$-connected. Since $a_{n,s}$ has degree $n$, we have that $\sigma^\ast(a_{n,s})=0$ which is a contradiction.
\end{proof}

Next we reduce to the case $n=2$.

\begin{lemma} \label{lem:reductionton2}
If $b_{2,s}$ is non-zero, then $b_{n,s}$ is non-zero for all $n\geq 2$.
\end{lemma}
\begin{proof}
Recall that there are maps $S^1\times \textnormal{SP}^{p^s}(S^n)\to \textnormal{SP}^{p^s}(S^{n+1})$ defined by
\[
(t,(x_1,\dots,x_{p^s}))\to (t\wedge x_1,\dots,t\wedge x_{p^s})
\]
which descend to maps
\[
\tau_n\co \Sigma \textnormal{SP}^{p^s}(S^n)\to \textnormal{SP}^{p^s}(S^{n+1})\, .
\]
It is easy to see that for each $n$, $\tau_n$ induces an isomorphism on $H^{n+1}$. Thus, we may assume $u_n$ be induced from a single element $u\in \varprojlim H^{n}(\textnormal{SP}^{p^s}(S^n))\cong \F_p$, so that $\tau_n^\ast(u_{n+1})$ corresponds to $u_n$ under the suspension isomorphism. The maps $\tau_n$ descend further to maps
\[
\bar{\tau}_n\co \Sigma  \overline{\textnormal{SP}}^{p^s}(S^n)\to  \overline{\textnormal{SP}}^{p^s}(S^{n+1})\, .
\]
It follows that also $v_n$ is stable, i.e., that $\bar{\tau}_n^\ast(v_{n+1})$ corresponds to $v_n$ under suspension. Now let $\mathfrak{z}\subseteq \mathfrak{u}_{p^s}$ denote the centre. Upon identifying $\mathfrak{z}^+$ with $S^1$, the map $\mathfrak{z}\times C_n(\mathfrak{u}_{p^s})\to C_{n+1}(\mathfrak{u}_{p^s})$ obtained by restricting the inclusion $\mathfrak{z}\times \mathfrak{u}_{p^s}^n\subseteq \mathfrak{u}_{p^s}^{n+1}$ induces a map
\[
\sigma_n\co \Sigma C_n(\mathfrak{u}_{p^s})^+ \to C_{n+1}(\mathfrak{u}_{p^s})^+\, .
\]
This map fits into a commutative diagram
\[
\xymatrix{
\Sigma C_n(\mathfrak{u}_{p^s})^+ \ar[r]^-{\Sigma \bar{\pi}_n} \ar[d]^-{\sigma_n} & \Sigma  \overline{\textnormal{SP}}^{p^s}(S^n) \ar[d]^-{\bar{\tau}_n} \\
C_{n+1}(\mathfrak{u}_{p^s})^+ \ar[r]^-{\bar{\pi}_{n+1}} &  \overline{\textnormal{SP}}^{p^s}(S^{n+1})\, .
}
\]
The diagram shows that if $\bar{\pi}_n^\ast(v_n)\neq 0$, then $\bar{\pi}_{n+1}^\ast(v_{n+1})\neq 0$. This gives the desired reduction to the case $n=2$.
\end{proof}

\begin{remark}
The $\{\bar{\tau}_n\}$ are the structure maps of the spectrum $\textnormal{SP}^{p^s}(\mathbb{S})/ \textnormal{SP}^{p^s-1}(\mathbb{S})$ where $\mathbb{S}$ is the sphere spectrum. The $\{v_n\}$ come from a class in spectrum cohomology
\[
v\in  \varprojlim H^{n+2(p^s-1)}(\overline{\textnormal{SP}}^{p^s}(S^n))\cong H^{2(p^s-1)}(\textnormal{SP}^{p^s}(\mathbb{S})/\textnormal{SP}^{p^s-1}(\mathbb{S}))\, .
\]
The proof showed that the spaces $\{C_n(\mathfrak{u}_{p^s})^+\}$ define a spectrum, too, with structure maps $\{\sigma_n\}$. The maps $\{\bar{\pi}_n\}$ give a map of spectra
\[
\bar{\pi}\co \{C_n(\mathfrak{u}_{p^s})^+\}\to \textnormal{SP}^{p^s}(\mathbb{S})/ \textnormal{SP}^{p^s-1}(\mathbb{S})
\]
which, when $s=1$, satisfies $\bar{\pi}^\ast(v)\neq 0$ as the next lemma will show.
\end{remark}

The following lemma shows that $b_{2,1}\neq 0$, and thus $b_{n,1}\neq 0$ for all $n\geq 2$ according to Lemma \ref{lem:reductionton2}. In view of Lemma \ref{lem:nonsplitting} this finishes the proof of Theorem \ref{thm:nonsplitting}.
\begin{lemma} \label{lem:maptosphere}
The map $\bar{\pi}_2\co C_2(\mathfrak{u}_p)^+ \to \overline{\textnormal{SP}}^{p}(S^2)\cong S^{2p}$ is an isomorphism in $H^{2p}$.
\end{lemma}
\begin{proof}
The homeomorphism $\overline{\textnormal{SP}}^{p}(S^2)\cong S^{2p}$ follows from Proposition \ref{prop:chevalley}. It also follows directly from the well-known fact that $\textnormal{SP}^{p}(\mathbb{C}P^1)\cong \mathbb{C}P^{p}$ (see e.g. \cite[Section 22]{Steenrod}).

Let $i\co (\mathfrak{t}^2)^+\to C_2(\mathfrak{u}_p)^+$ be the map induced by the inclusion $\mathfrak{t} \subseteq \mathfrak{u}_p$. To prove the lemma it will be enough to show that
\[
i^\ast\co H^{2p}(C_2(\mathfrak{u}_p)^+)\to H^{2p}((\mathfrak{t}^2)^+)
\]
is trivial with mod-$p$ coefficients. Indeed, consider the composite of $i$ and $\bar{\pi}_2$ on cohomology with $\Z_{(p)}$- and $\F_p$-coefficients. By Theorem \ref{thm:char0} and Lemma \ref{lem:degree2p} we have that $H^{2p}(C_2(\mathfrak{u}_p)^+;\Z_{(p)})\cong \Z_{(p)}$ and $H^{2p}(C_2(\mathfrak{u}_p)^+;\F_p)\cong \F_p$, so that we have a diagram of the following form:

\[
\xymatrix{
\Z_{(p)} \ar[r]^-{\bar{\pi}_2^\ast}_-{\cong} \ar[d] & \Z_{(p)} \ar[r]^-{i^\ast} \ar[d] & \Z_{(p)} \ar[d] \\
\F_p \ar[r]^-{\bar{\pi}_2^\ast} & \F_p \ar[r]^-{i^\ast}_-{0} & \F_p 
}
\]
Here the vertical maps are reduction modulo $p$. If we assume that the lower right map is zero, then the degree of $i^\ast\co \Z_{(p)}\to \Z_{(p)}$ must be divisible by $p$. Since the composite of the top horizontal maps is multiplication by $p!$ (see Remark \ref{rem:degree}), the top left map must be an isomorphism. It follows from the diagram that the lower left map is an isomorphism, too.

It remains to show that $i^\ast$ is trivial with $\F_p$-coefficients.  Let $S\subseteq SU(p)$ be a maximal torus and $\mathfrak{s}$ its Lie algebra, so that $\mathfrak{t}\cong \mathfrak{z}\oplus \mathfrak{s}$. By Proposition \ref{prop:centresuspension} the map $i$ is the double suspension of the inclusion $(\mathfrak{s}^2)^+\to C_2(\mathfrak{su}_p)^+$. By Lemma \ref{lem:suspension}, the latter is the (unreduced) suspension of the inclusion
\[
j \co S(\mathfrak{s}^2)\to C_2^1(\mathfrak{su}_p)\, ,
\]
where $S(\mathfrak{s}^2)\subseteq \mathfrak{s}^2$ is the unit sphere and $C_2^1(\mathfrak{su}_p)$ is the intersection of $C_2(\mathfrak{su}_p)$ with the unit sphere in $\mathfrak{su}_p^2$. Thus, we may equivalently show that $j^\ast$ is trivial with $\F_p$-coefficients.

The inclusion $j$ can be factored as
\[
S(\mathfrak{s}^2) \to SU(p)/S\times_{\Sigma_p} S(\mathfrak{s}^2)  \xrightarrow{\phi'} C_2^1(\mathfrak{su}_p)\,,
\]
where $\phi'$ is the action map defined in (\ref{eq:conjugationmapunitsphere}) and the first map is the inclusion of the fibre of the sphere bundle
\begin{equation} \label{eq:spherebundle0}
SU(p)/S\times_{\Sigma_p} S(\mathfrak{s}^2)\to (SU(p)/S)/\Sigma_p\,.
\end{equation}
We claim that the fibre inclusion is zero in mod-$p$ cohomology. Let $(E_r^{\ast,\ast},d_r)$ be the Serre spectral sequence of (\ref{eq:spherebundle0}). Then $E^{0,2p-3}_2\cong \F_p$ and $d_2\co E^{0,2p-3}_2\to E^{2p-2,0}_2$ is determined by the mod-$p$ Euler class of (\ref{eq:spherebundle0}). In Lemma \ref{lem:eulerclass} we show that the Euler class is non-zero, so $d_2$ is injective and thus $E_\infty^{0,2p-3}=0$. This proves the claim.
\end{proof}

\begin{remark} \label{rem:casesup}
Let us explain why the same line of arguments followed in this section fails to work for $SU(p)$. In the case of $\Hom(\Z^2,U(p))$ diagram (\ref{dgr:symmetricproduct}) takes the form
\[
\xymatrix{
\Hom(\Z^2,U(p)) \ar[r] \ar[d] & \C P^p \ar[d] \\
C_2(\mathfrak{u}_p)^+\ar[r] & S^{2p}
}
\]
where the right vertical map is the projection collapsing $\C P^{p-1}$. The obstruction for splitting the left vertical map comes from the $p$-th power operation
\[
\mathscr{P}^1\co H^2(\C P^p)\to H^{2p}(\C P^{p})\, .
\]
In contrast, for $SU(p)$ the moduli space of representations is
\[
\Hom(\Z^2,SU(p))/SU(p)\cong \C P^{p-1}
\]
(see e.g. \cite[Proposition 6.5]{ACG}), and $\C P^{p-1}$ does not have any non-trivial mod-$p$ Steenrod operations. Indeed, when $p=2$ we know that $C_2(\mathfrak{su}_p)^+$ is a stable summand of $\Hom(\Z^2,SU(p))_+$.

\end{remark}

 \appendix \section{Topology of commuting varieties}\label{sec:commutingvariety}

Suppose that $G$ is a compact connected Lie group of rank $r\geq 1$ and dimension $d$ and let 
$\mathfrak{g}$ be its Lie algebra. In this appendix we make some basic observations about the 
one-point compactification of the commuting variety
\[
C_n(\mathfrak{g})=\{(X_1,\dots,X_n)\in \mathfrak{g}^n\mid [X_i,X_j]=0\textnormal{ for all }1\leq i,j\leq n\}
\]
which we view as a subspace of the Euclidean space $\mathfrak{g}^n\cong \R^{nd}$. As before, we let $G$ act on $C_n(\mathfrak{g})^+$ 
diagonally via the adjoint representation, leaving invariant the point $+$ at infinity.

\subsection{General results} \label{sec:general} We start by observing that the $G$-action on $C_n(\mathfrak{g})^+$ has isotropy groups 
of maximal rank, i.e., the isotropy group at any point of $C_n(\mathfrak{g})^+$ contains a maximal torus of $G$.

\begin{proposition} \label{prop:cartansubalgebra}
Any $(X_{1},\dots,X_{n})\in C_{n}(\g)$ is contained in a Cartan subalgebra of $\g$.
\end{proposition}
\begin{proof}
Suppose that $(X_{1},\dots,X_{n})\in C_{n}(\g)$ and let $g_{i}:=\exp(X_{i})$ for
every $1\le i\le n$. Note that $(g_{1},\dots,g_{n})\in
\Hom(\Z^{n},G)_{\BONE}$, the path--component of the space 
of commuting $n$-tuples $\Hom(\Z^{n},G)$ that contains the trivial representation $\BONE=(1,\dots,1)$.  Indeed, the Baker--Campbell--Hausdorff
formula shows that for every $1\le i,j\le n$
\[
g_{i}g_{j}=\exp(X_{i}+X_{j})=\exp(X_{j}+X_{i})=g_{j}g_{i}.
\]
Now the map $\gamma\co [0,1]\to \Hom(\Z^{n},G)$ given by $t\mapsto
(\exp(tX_{1}),\dots,\exp(tX_{n}))$ provides a path from $\BONE$ to $(g_{1},\dots,g_{n})$. In
particular, for every $t\in [0,1]$ the $n$-tuple
$(\exp(tX_{1}),\dots,\exp(tX_{n}))$ is contained in a maximal
torus of $G$, by \cite[Lemma 4.2]{Baird}. Choose $t>0$ small enough such that, for all $1\leq i \leq n$, $t X_i$ lies in a neighbourhood of $0\in \mathfrak{g}$ on which the exponential map is injective. Let $T$ be a maximal
torus containing $(\exp(tX_{1}),\dots,\exp(tX_{n}))$. Let $\t\subseteq
\g$ be the Lie algebra of $T$. It follows
that $tX_{i}\in \t$ and thus $X_{i}\in \t$ for all $1\le i\le n$. We
conclude that any $\underline{X}:=(X_{1},\dots,X_{n})\in C_{n}(\g)$
is contained in some Cartan subalgebra of $\g$.
\end{proof}

\begin{corollary} \label{cor:maxrankisotropy}
The $G$-action on $C_n(\mathfrak{g})^+$ has isotropy groups of maximal rank.
\end{corollary}
\begin{proof}
The point $+$ at infinity is fixed by $G$. By Proposition \ref{prop:cartansubalgebra}, any $\underline{X}=(X_1,\dots,X_n)\in C_n(\mathfrak{g})$ lies in a Cartan subalgebra $\mathfrak{t}$, which we view as the Lie algebra of a maximal torus $T\subseteq G$. Then the isotropy
subgroup 
\[
G_{\underline{X}}=\left\{g\in G~|~ \textnormal{Ad}_{g}(X_{i})=X_{i} \,,
\text{ for all } i=1,\dots, n\right\}
\]
where $\textnormal{Ad}_g\co \mathfrak{g}\to \mathfrak{g}$ is the adjoint action of $g\in G$, contains the maximal torus $T$. 
\end{proof}

Next we consider the homology of $C_n(\mathfrak{g})^+$. Let $T\subseteq G$ be a maximal torus, $\mathfrak{t}\subseteq \g$ its Lie algebra, and $W$ the Weyl group relative to $T$. Consider the vector bundle
\[
\mathfrak{t}^n\to G/T\times_{W}\mathfrak{t}^n\to (G/T)/W
\]
associated with the $W$-representation $\mathfrak{t}^n=\mathfrak{t}\oplus \cdots \oplus \mathfrak{t}$. The map
\begin{alignat*}{1}
\phi\co G/T\times_W \mathfrak{t}^{n} & \to C_{n}(\mathfrak{g}) \\
[gT,X_1,\dots,X_{n}] & \mapsto (\textnormal{Ad}_g(X_1),\dots,\textnormal{Ad}_g(X_{n}))\,,
\end{alignat*}
where $\textnormal{Ad}_g\co \mathfrak{g}\to \mathfrak{g}$ is the adjoint action of $g\in G$, is proper. Therefore it extends to a map from the Thom space
\[
\phi^+\co (G/T)_+\wedge_W (\mathfrak{t}^n)^+\to C_{n}(\mathfrak{g})^+\, .
\]
Let $\tilde{\phi}\co G/T\times_W (\mathfrak{t}^n)^+\to C_n(\mathfrak{g})^+$ be the composition of $\phi^+$ and the projection
\[
G/T\times_W (\mathfrak{t}^n)^+\to (G/T)_+\wedge_W (\mathfrak{t}^n)^+
\]
that collapses the section at infinity.

\begin{lemma} \label{lem:actionmapliealgebra}
The map $\tilde{\phi} \co G/T\times_W (\mathfrak{t}^n)^+\to C_n(\mathfrak{g})^+$ induces an isomorphism on homology with coefficients in $\Z[1/|W|]$, and so does the map $\phi^+$.
\end{lemma}
\begin{proof}
By Corollary \ref{cor:maxrankisotropy} the isotropy group of any point of $C_{n}(\mathfrak{g})^+$ under the adjoint action has maximal rank. Moreover, $(C_n(\mathfrak{g})^+)^T=(\mathfrak{t}^n)^+$. Thus, it follows from \cite[Theorem 3.3]{Baird} that $\tilde{\phi}$ induces an isomorphism with coefficients in $\Z[1/|W|]$. Furthermore, the homology of $(G/T)/W$ with $\Z[1/|W|]$-coefficients is trivial, so the map collapsing the section at infinity is a homology isomorphism as well.
\end{proof}

\begin{theorem} \label{thm:char0}
Let $n\geq 1$. Then $C_{n}(\mathfrak{g})^+$ is a $\Z[1/|W|]$-homology sphere of dimension $nr$ if $n$ is even, and of dimension $d+(n-1)r$ if $n$ is odd.
\end{theorem}
\begin{proof}
Let $k=\Z[1/|W|]$. According to Lemma \ref{lem:actionmapliealgebra} there is an isomorphism
\[
H^\ast(C_{n}(\mathfrak{g})^+;k)\cong H^\ast(G/T\times (\mathfrak{t}^{n})^+;k)^{W}\, .
\]

If $n$ is even, then $W$ acts orientation-preservingly on $(\mathfrak{t}^n)^+$ and thus trivially on $H^\ast((\mathfrak{t}^n)^+;k)$. As an ungraded $W$-module $H^{*}(G/T;k)$ is isomorphic to the group ring $k W$, and thus the trivial representation $H^\ast(G/T;k)^W$ is one-dimensional and must occur in degree $0$. This shows that
\[
H^i(C_n(\mathfrak{g})^+;k)\cong \begin{cases} k & \textnormal{if }i=0,nr \\ 0 & \textnormal{otherwise,} \end{cases}
\]
and hence that $C_n(\mathfrak{g})^+$ is a $k$-homology sphere of dimension $nr$.

On the other hand, if $n$ is odd, then $W$ acts orientation-reversingly on $(\mathfrak{t}^n)^+$, and so $H^\ast((\mathfrak{t}^n)^+;k)$ is the sum of the trivial $W$-module $k$ in degree zero and the sign representation $k^\sigma$ in degree $nr$. Since $|W|$ is invertible in $k$, there is an isomorphism
\[
\Bigl(H^\ast(G/T;k)\otimes k^\sigma \Bigr)^W\cong H^\ast((G/T)/W;k^\sigma)
\]
and, by Poincar{\'e} duality, this is $k$ in degree $d-r$ and zero in all other degrees. Taking into account the fact that $k^\sigma$ is in degree $nr$, it follows that
\[
\Bigl(H^\ast(G/T;k)\otimes H^\ast((\mathfrak{t}^n)^+;k) \Bigr)^W
\]
is $k$ in degrees $0$ and $d+(n-1)r$, and zero otherwise. This finishes the argument when $n$ is odd.
\end{proof}

Since $G$ is compact, we can fix an Ad-invariant inner product on $\mathfrak{g}^n$. Let $S(\mathfrak{g}^n)\subseteq \mathfrak{g}^n$ denote the unit sphere with respect to this inner product. Let
\[
C^1_n(\mathfrak{g}):= C_n(\mathfrak{g})\cap S(\mathfrak{g}^n)\, .
\]
We use the superscript $\lozenge$ to denote the unreduced suspension.

\begin{lemma} \label{lem:suspension}
There is a $G$-equivariant homeomorphism
\[
C_n(\mathfrak{g})^+\cong C_n^1(\mathfrak{g})^{\lozenge}
\]
where $G$ acts trivially on the suspension coordinate.
\end{lemma}
\begin{proof}
Fix a homeomorphism $\gamma \co [0,1]\stackrel{\cong}{\to} \R_{\geq 0}^+$ such that $\gamma(0)=0$ and $\gamma(1)=+$. Using this define a $G$-equivariant homeomorphism

\[
S(\mathfrak{g}^n)^{\lozenge}  \stackrel{\cong}{\to} (\mathfrak{g}^n)^+\,,\quad  (t,X)  \mapsto \begin{cases} \gamma(t)X & \textnormal{if } t<1 \\ + & \textnormal{if } t=1. \end{cases}
\]
Since rescaling preserves commutativity, this restricts to a $G$-equivariant homeomorphism $C_n^1(\mathfrak{g})^{\lozenge}\cong C_n(\mathfrak{g})^+$.
\end{proof}

Let us view $S(\mathfrak{t}^n)\subseteq (\mathfrak{t}^n)^+$ as the equatorial sphere. The map $\tilde{\phi} \co G/T\times_W (\mathfrak{t}^n)^+\to C_n(\mathfrak{g})^+$ then restricts to a map
\begin{equation} \label{eq:conjugationmapunitsphere}
\phi'\co G/T\times_W S(\mathfrak{t}^n) \to C_n^1(\mathfrak{g})\, .
\end{equation}

\begin{proposition} \label{prop:simplyconnected}
For every $n\geq 1$ the space $C_n(\mathfrak{g})^+$ is simply-connected, unless $n=1$ and $\dim(\mathfrak{g})=1$ in which case $C_n(\mathfrak{g})^+\cong S^1$.
\end{proposition}
\begin{proof}
First assume $n=1$. Then $C_n(\mathfrak{g})^+\cong \mathfrak{g}^+\cong S^{\dim(\mathfrak{g})}$ which is simply-connected, unless $\dim(\mathfrak{g})=1$. Now assume $n\geq 2$. By Proposition \ref{prop:cartansubalgebra} the map $\phi'$ is surjective and hence, since $S(\mathfrak{t}^n)$ is path-connected, $C^1_n(\mathfrak{g})$ is path-connected as well. Since $C_n(\mathfrak{g})^+$ is the suspension of $C^1_n(\mathfrak{g})$, it is thus simply-connected.
\end{proof}

Next we describe, for the particular case of commuting pairs, the quotient of $C_{2}(\mathfrak{g})^+$ by the adjoint action of $G$.

\begin{proposition} \label{prop:chevalley}
Let $G$ be a compact connected Lie group of rank $r$. Then there is a homeomorphism $C_2(\mathfrak{g})^+/G\cong S^{2r}$.
\end{proposition}
\begin{proof}
In Proposition \ref{prop:cartansubalgebra} we showed that every $(X_1,X_2)\in C_2(\mathfrak{g})$ is contained in a Cartan subalgebra $\mathfrak{t}\subseteq \mathfrak{g}$. Therefore, the quotient $C_2(\mathfrak{g})^+/G$ may be identified with the quotient of $(\mathfrak{t}^2)^+$ by the diagonal action of the Weyl group. We can identify the $W$-module $\mathfrak{t}^{2}$ with $\mathfrak{t}\otimes {\C}$, the complexification of $\mathfrak{t}$. By Chevalley's theorem the ring of $W$-invariant polynomial functions on $\mathfrak{t}\otimes \C$ is a polynomial algebra generated by $r$ homogeneous elements. These homogeneous elements provide a system of global coordinates defining a homeomorphism of $(\mathfrak{t}\otimes \C)/W$ with $\C^{r}$ (cf. \cite[Prop. 9.3]{LT}), and hence a homeomorphism of $(\mathfrak{t}\otimes \C)^+/W$ with $(\C^{r})^+\cong S^{2r}$.
\end{proof}

\begin{remark} \label{rem:degree}
The composition of the inclusion $(\mathfrak{t}^2)^+\to C_2(\mathfrak{g})^+$ with the projection $C_2(\mathfrak{g})^+\to C_2(\mathfrak{g})^+/G$ may be identified with the projection $(\mathfrak{t}^2)^+\to (\mathfrak{t}^2)^+/W$. By computing local degrees one sees that this map has degree $|W|$.
\end{remark}

Finally, using the classification of compact Lie algebras, it follows that
\[
\g\cong \z\oplus \g_{1}\oplus\cdots \oplus \g_{s}\,,
\] 
where $\z$ denotes the center of $\g$ and $\g_{1},\dots \g_{s}$ are simple Lie algebras. From here we get the following statement.
\begin{proposition} \label{prop:centresuspension}
For all $n\geq 0$, $C_{n}(\g)^{+}\cong \Sigma^{n\dim(\z)}\left(C_{n}(\g_{1})^{+}\wedge \cdots \wedge C_{n}(\g_{s})^{+} \right)$ where the $\mathfrak{g}_i$ are the simple factors of $\mathfrak{g}$.
\end{proposition}
\begin{proof}
By definition of the Lie bracket on a direct sum of Lie algebras we have that
\[
C_n(\mathfrak{g})\cong C_n(\mathfrak{z})\times C_n(\mathfrak{g}_1)\times \cdots \times C_n(\mathfrak{g}_s)\, .
\]
Since $\mathfrak{z}$ is abelian, we further have $C_n(\mathfrak{z})=\mathfrak{z}^n$, and the statement follows.
\end{proof}

\begin{remark}
We clarify the comment made in the introduction that $C_n(\mathfrak{g})$ is a possibly infinite subspace arrangement. The space $C_n(\mathfrak{g})\subseteq \mathfrak{g}^n$ can be written, tautologically, as a union of lines in $\mathfrak{g}^n$. However, unless $\mathfrak{g}$ is abelian, it cannot be written as a union of \emph{finitely} many linear subspaces of $\mathfrak{g}^n$. For example, $C_2(\mathfrak{su}_2)\subseteq \R^6$ is the determinantal variety of real $2\times 3$ matrices of rank less than two. Elementary linear algebra shows that a linear subspace $V\subseteq \R^6$ which is contained in $C_2(\mathfrak{su}_2)$ can have dimension at most three. However, with the singular point deleted, $C_2(\mathfrak{su}_2)$ is a $4$-dimensional manifold, namely the complement of the zero section in a rank two vector bundle over $\R P^2$ (see Example \ref{ex:su2}). More generally, since every commuting $n$-tuple in $\mathfrak{g}$ lies in a Cartan subalgebra of $\mathfrak{g}$, an argument like the one in \cite[Corollary 2.5]{Richardson} shows that $C_n(\mathfrak{g})$ is an irreducible real algebraic variety, and thus it cannot be written as a union of finitely many proper linear subspaces of $\mathfrak{g}^n$.
\end{remark}

\subsection{Mod-$p$ cohomology of $C_n(\mathfrak{u}_p)^+$ for $n>2$ even} \label{sec:cohomologycnup}

Let $p$ be a prime. To lighten notation we will omit the coefficient ring and all cohomology groups in this section are understood with $\F_p$-coefficients.  For a space $X$ with mod-$p$ cohomology of finite type, define the mod-$p$ Poincar{\'e} series by
\[
\Pi_X^p(t)=\sum_{k\geq 0} \dim_{\F_p}H^k(X)\, t^k\, .
\]
The objective of this section is to prove the following theorem.

\begin{theorem} \label{thm:commutingvarietymodp}
For every prime $p$ and every even integer $n\geq 4$
\[
\Pi_{C_{n}(\mathfrak{u}_p)^+}^p(t)= 1+t^{np}+(t^{n+1}+t^{np})(t^{2p-3}+t^{2p-2})\prod_{k=1}^{p-2}(1+t^{2k-1})\, .
\]
\end{theorem}

Our computation is based on a recent computation due to Guerra and Jana \cite{GJ} of the mod-$p$ cohomology of the unordered complete flag manifold in $\C^p$ which we restate here for convenience. Let $T\subseteq U(p)$ be a maximal torus and $N$ its normaliser.
\begin{theorem}[{\cite[Theorem 6.8]{GJ}}] \label{thm:flagvariety} If $p$ is an odd prime, there is an isomorphism of graded $\F_p$-vectorspaces
\[
H^\ast(U(p)/N;\F_p)\cong  \F_p\{1,\alpha_S,\gamma_S\,|\, S\subset \{1,\dots,p-2\}\}\,,
\]
where $S$ is allowed to be empty, and the degrees of the generators are
\[
|\alpha_S|=2p-3+\sum_{i\in S} (2i-1)\,,\quad |\gamma_S|=2p-2+\sum_{i\in S} (2i-1)\, .
\]
\end{theorem}
It follows that the mod-$p$ Poincar{\'e} series of $U(p)/N$ is (see \cite[Corollary 6.9]{GJ})
\[
\Pi^p_{U(p)/N}(t)=1+(t^{2p-3}+t^{2p-2})\prod_{k=1}^{p-2} (1+t^{2k-1})\, .
\]
This formula remains valid when $p=2$ in which case $U(2)/N\cong \R P^2$.

To compute the cohomology of $C_n(\mathfrak{u}_p)^+$ we first recall from Lemma \ref{lem:suspension} and Proposition \ref{prop:centresuspension} that
\[
C_n(\mathfrak{u}_p)^+ \cong \Sigma^n C_n(\mathfrak{su}_p)^+\cong \Sigma^n C_n^1(\mathfrak{su}_p)^{\lozenge}\, ,
\]
so it will suffice to compute the cohomology of $C_n^1(\mathfrak{su}_p)$. This will be achieved by an argument similar to the one used in the proof of Lemma \ref{lem:actionmapliealgebra}.

Let $S\subseteq SU(p)$ be a maximal torus with Lie algebra $\mathfrak{s}\subseteq \mathfrak{su}_p$, and write $S(\mathfrak{s}^n)\subseteq \mathfrak{s}^n$ for the unit sphere.

\begin{lemma} \label{lem:reducedconjugationmap}
For every $n\geq 1$ the map defined in (\ref{eq:conjugationmapunitsphere})
\[
\phi' \co SU(p)/S\times_{\Sigma_p} S(\mathfrak{s}^{n}) \to C_{n}^1(\mathfrak{su}_p)
\]
induces an isomorphism on mod-$p$ cohomology.
\end{lemma}
\begin{proof}
This is an application of the Vietoris-Begle mapping theorem in the form stated, \emph{e.g.}, in \cite[Theorem 2.1]{Baird}. Since $\phi'$ is induced by the restriction of an action map of a compact group, it is closed. As every $(X_1,\dots,X_n)\in C_n^1(\mathfrak{su}_p)$ is contained in a Cartan subalgebra of $\mathfrak{su}_p$, the map $\phi'$ is surjective. Then we need only show that $\phi'$ has $\F_p$-acyclic fibres.

Let $\underline{X}=(X_1,\dots,X_n)\in C_n^1(\mathfrak{su}_p)$. Without loss of generality we may assume that $\underline{X}\in S(\mathfrak{s}^n)$ and that $X_1\neq 0$. It is easy to check (see \cite[Lemma 3.2]{Baird}) that
\[
(\phi')^{-1}(\underline{X})\cong SU(p)_{\underline{X}}^0/N_{SU(p)_{\underline{X}}^0}(S)\,,
\]
where $SU(p)_{\underline{X}}^0$ is the connected component of the identity of the isotropy group of $\underline{X}$ (in fact, these isotropy groups are connected, but we shall not need this here). The Weyl group of $SU(p)_{\underline{X}}^0$ is contained in the isotropy group of $X_1\in \mathfrak{s}$ for the action of $\Sigma_p$ on $\mathfrak{s}$. As $X_1\neq 0$, this isotropy group is a proper Young subgroup of $\Sigma_p$ and thus has order prime to $p$. It follows that $(\phi')^{-1}(\underline{X})$ is $\F_p$-acyclic, by the usual argument that the cohomology of the quotient of a compact connected Lie group by the normaliser of a maximal torus is $k$-acyclic for any field $k$ whose characteristic is prime to the order of the Weyl group.
\end{proof}

Now we consider the fibre bundle
\begin{equation} \label{eq:spherebundle}
S(\mathfrak{s}^n)\to SU(p)/S\times_{\Sigma_p} S(\mathfrak{s}^n) \to (SU(p)/S)/\Sigma_p\,,
\end{equation}
whose fibre is a sphere of dimension $n(p-1)-1$ and whose base may be identified with $U(p)/N$. When $n$ is even, the bundle is orientable. In this case, the mod-$p$ cohomology of the total space, and hence that of $C_n^1(\mathfrak{su}_p)$, can be computed using the Gysin exact sequence
\[
\cdots \to H^k(U(p)/N) \xrightarrow{\cdot \chi_n} H^{k+n(p-1)}(U(p)/N) \to H^{k+n(p-1)}(E)\to H^{k+1}(U(p)/N) \xrightarrow{\cdot \chi_n} \cdots\, .
\]
Here $E=SU(p)/S\times_{\Sigma_p} S(\mathfrak{s}^n)$ is the total space and $\chi_n\in H^{n(p-1)}(U(p)/N)$ the mod-$p$ Euler class of (\ref{eq:spherebundle}). 

Next, we determine $\chi_n$. Let $C_p\subseteq \Sigma_p$ be a Sylow-$p$ subgroup and let $x\in H^2(C_p)\cong \F_p$ be a generator. The inclusion induces an isomorphism
\[
H^{2p-2}(\Sigma_p)\cong H^{2p-2}(C_p)
\]
by which we may view $x^{p-1}$ as a generator of $H^{2p-2}(\Sigma_p)$. Since $\Sigma_p$ is the fundamental group of $U(p)/N$, $x^{p-1}$ determines an element in $H^{2p-2}(U(p)/N)$. Inspection of the proof of Theorem \ref{thm:flagvariety} shows that this element is, up to a unit, the generator $\gamma_\emptyset$ (see \cite[p. 4]{GJ} where $\gamma_{\emptyset}$ is referred to as $\gamma_{1,1}$).

\begin{lemma} \label{lem:eulerclass}
For the mod-$p$ Euler class of (\ref{eq:spherebundle}) we have
\[
\chi_n=\begin{cases} u\gamma_{\emptyset} & \textnormal{if }n=2, \\ 0 & \textnormal{if } n>2 \end{cases}
\]
where $u\in \mathbb{F}_p^\times$.
\end{lemma}
\begin{proof}
From the viewpoint of representation theory $\mathfrak{s}$ is the $(p-1)$-dimensional real standard representation of $\Sigma_p$. Since $n$ is even, we have that
\[
\mathfrak{s}^n \cong (\mathfrak{s} \otimes \C)^{n/2}
\]
as $\Sigma_p$-representations. Now $\chi_n$ is the Euler class of the vector bundle over $U(p)/N$ defined by the representation $\mathfrak{s}^n$, thus
\[
\chi_n=\chi_2^{n/2}\, .
\]

To show that $\chi_2$ equals $\gamma_{\emptyset}$ up to a unit, it is enough to show, by the discussion preceding the lemma, that the Euler class in $H^{2p-2}(C_p)$ associated with the restriction of $\mathfrak{s} \otimes \C$ to a $C_p$-representation is $u x^{p-1}$ for some $u\in \F_p^\times$. But it is well-known that this restriction is the direct sum of $p-1$ non-trivial one-dimensional $C_p$-representations, each of which has Euler class a non-zero multiple of the generator $x\in H^2(C_p)$. 

Now suppose $n>2$. To show that $\chi_n=0$  it is enough to show that $\gamma_\emptyset^2=0$ in $H^\ast(U(p)/N)$. The proof of Theorem \ref{thm:flagvariety} shows that the image of the natural map $H^\ast(BN)\to H^\ast(U(p)/N)$ is $\F_p[\alpha_\emptyset, \gamma_\emptyset]/(\gamma_\emptyset^2,\alpha_\emptyset^2,\alpha_\emptyset\gamma_\emptyset)$. Since this is a subring, it follows that $\gamma_\emptyset^2=0$.
\end{proof}

\begin{proof}[Proof of Theorem \ref{thm:commutingvarietymodp}]
We consider the Gysin sequence displayed above. If $n> 2$, then $\chi_n=0$ by Lemma \ref{lem:eulerclass}, and we deduce for all $k\geq 0$
\[
\tilde{H}^k(E)\cong \tilde{H}^{k}(U(p)/N)\oplus H^{k-n(p-1)+1}(U(p)/N)\, .
\]
Thus, by Lemma \ref{lem:reducedconjugationmap} and the remarks preceding it, we obtain for the reduced Poincar{\'e} polynomial $\tilde{\Pi}_{C_n(\mathfrak{u}_p)^+}^p(t):=\Pi_{C_n(\mathfrak{u}_p)^+}^p(t)-1$
\[
\tilde{\Pi}_{C_n(\mathfrak{u}_p)^+}^p(t)=t^{n+1}\, \tilde{\Pi}^p_{E}(t)=t^{n+1}\,\tilde{\Pi}^p_{U(p)/N}(t)+t^{np}\,\Pi_{U(p)/N}^p(t)\, .
\]
Plugging in the Poincar{\'e} series for $U(p)/N$ yields the desired result.
\end{proof}

The situation for $n=2$ is more subtle. The proof of Theorem \ref{thm:flagvariety} given in \cite{GJ} shows that multiplication by the Euler class $\chi_2=\gamma_\emptyset$ annihilates positive degree elements in the $E_\infty$-page of the spectral sequence computing $H^\ast(U(p)/N)$. However, there may be multiplicative extensions in the spectral sequence, which is why we cannot conclude that multiplication by $\chi_2$ annihilates positive degree elements in $H^\ast(U(p)/N)$. However, since $|\gamma_{\emptyset}|=2p-2$ and $|\alpha_{\emptyset}|=2p-3$, the class $\alpha_{\emptyset}\in H^{2p-3}(U(p)/N)$ survives the Gysin sequence showing:
\begin{lemma} \label{lem:degree2p}
For any prime $p$, $H^{2p}(C_2(\mathfrak{u}_p)^+;\F_p)\cong \F_p$.
\end{lemma}

\begin{remark}
The mod-$p$ cohomology of $C_n(\mathfrak{u}_p)^+$ for odd $n$ is not computed as easily; in this case there is no Euler class and the differentials in the Serre spectral sequence associated with (\ref{eq:spherebundle}) must be determined by other means. These differentials are non-trivial as can be seen in the simplest case $n=1$.
\end{remark}

\end{document}